\documentclass[12pt,reqno]{amsart}

\usepackage{a4wide,color}
\usepackage[english]{babel}
\usepackage[utf8]{inputenc}

\usepackage{bbm,amsmath,amsthm,epsfig,latexsym,marvosym}
\usepackage{amsfonts}
\usepackage{amsmath}
\usepackage{amssymb}
\usepackage{bigints}
\usepackage{extarrows}
\usepackage{graphicx}
\usepackage{mathrsfs} 
\usepackage{subfig}
\usepackage{esint,xfrac}
\usepackage[colorlinks=true,linkcolor=blue,citecolor=red]{hyperref}

\allowdisplaybreaks[4]
\newtheorem{theorem}{Theorem}[section]
\newtheorem{corollary}[theorem]{Corollary}
\newtheorem{proposition}[theorem]{Proposition}
\newtheorem{lemma}[theorem]{Lemma}

\newtheorem{conjecture}[theorem]{Conjecture}
\newtheorem{remark}[theorem]{Remark}

\theoremstyle{definition}

\numberwithin{equation}{section}
\numberwithin{figure}{section}

\title{Jerison-Lee identities and Semi-linear subelliptic equations on CR manifolds}

\author{Xi-Nan Ma}

\address{School of Mathematical Sciences, University of Science and Technology of China, Hefei, 230026, Anhui Province, China}

\email{xinan@ustc.edu.cn}

\author{Qianzhong Ou}

\address{School of Mathematics and Statistics, Guangxi Normal University, Guilin, 541004, Guangxi Province, China}

\email{ouqzh@gxnu.edu.cn}

\author{Tian Wu}

\address{School of Mathematical Sciences, University of Science and Technology of China, Hefei, 230026, Anhui Province, China}

\email{wt1997@mail.ustc.edu.cn}

\thanks{The authors were supported by the National Natural Science Foundation of China (grants 12141105) and the first author also was supported by the National Key Research and Development Project (grants SQ2020YFA070080).}

\subjclass[2020]{Primary 32V20; Secondary 35J61}
\keywords{Cauchy-Riemann Yamabe problem, subelliptic equations, Folland-Stein inequality}

\begin{document}
\maketitle
\begin{abstract}
  In the study of the extremal for Sobolev inequality on the Heisenberg group and the Cauchy-Riemann(CR) Yamabe problem, Jerison-Lee found a three-dimensional family of differential identities for critical exponent subelliptic equation on Heisenberg group $\mathbb H^n$ by using the computer in \cite{MR0924699}. They wanted to know whether there is a theoretical framework that would predict the existence and the structure of such formulae. With the help of dimensional conservation and invariant tensors, we can answer the above question. For a class of subcritical exponent subelliptic equations on the CR manifold, several new types of differential identities are found. Then we use those identities to get the rigidity result, where rigidity means that subelliptic equations have no other solution than some constant at least when parameters are in a certain range. The rigidity result also deduces the sharp Folland-Stein inequality on closed CR manifolds.
\end{abstract}
\tableofcontents
\section{Introduction}

 Let $M$ be a real manifold. A distinguished complex subbundle $T^{(1,0)}M$ of $\mathbb CTM$ is a CR structure, if $T^{(1,0)}\cap T^{(0,1)}=0$, where $T^{(0,1)}:=\overline {T^{(1,0)}}$, and $M$ is called a CR manifold. An CR manifold $M$ is  hypersurface type, if $\dim_{\mathbb R}M=2n+1$ and $\dim_{\mathbb C}T^{(1,0)}M=n$.

If $M$ is oriented, a globally defined real one-form $\theta$ that annihilates $T^{(1,0)}M$ and $T^{(0,1)}M$ exists. The Levi form $\langle V,W\rangle_{L_\theta}=L_\theta(V,\overline W):=-2\sqrt{-1}\mathrm d\theta(V\wedge\overline W)$ is a hermitian form on $T^{(1,0)}M$. We say the CR structure is strictly pseudoconvex if $L_\theta$ is positive definite on $T^{(1,0)}M$ for some choice of $\theta$, in which case $\theta$ defines a contact structure on $M$ and we call $\theta$ a contact form associated with the CR structure.

The Reeb vector field $T$ is defined by $\theta(T)=1$, and $\mathrm d\theta(T,X)=0$, $\forall X\in TM$, then $TM=T^{(1,0)}M\oplus T^{(0,1)}M\oplus\operatorname{span}\{T\}$. Let $\{Z_i\}_{i=1}^n$ be an orthogonal basis w.r.t. the Levi form, and $\{\theta^i\}_{i=1}^n$ be the dual base of $\{Z_i\}_{i=1}^n$, then $\theta^i(T)=0$. Set $\mathrm d\theta=2\sqrt{-1}h_{i\overline j}\theta^i\wedge\theta^{\overline j}$, and we'll use the hermitian matrix $h_{i\overline j}$ and its inverse $h^{i\overline j}$ to raise and lower indices. Let $R_{i\overline jk\overline l}$ be the Webster curvature tensor , $\operatorname{Tor}(Z_i,Z_j)=A_{ij}$ the Webster torsion tensor, $\operatorname{Ric}(Z_i,Z_j)=R_{i\overline j}=R_{i\ k\overline j}^{\ k}$ the pseudohermitian Ricci tensor, and $R=R_i^{\ i}$ the pseudohermitian scalar curvature.

In this article, all small English letters in lower or upper place will be considered as summation indices taking part in the process of summing from 1 to $n$. Besides, all Greek letters and Arabic numbers in the lower or upper place won't participate in the summation process. Denote $Z_i f$ as $f_{,i}$, $Z_{\overline i} f$ as $f_{,\overline i}$, $Tf$ as $f_{,0}$. Commutation formulae are presented as follows:

$$f_{,ij}=f_{,ji},\ f_{,i\overline{j}}-f_{,\overline{j}i}=2\sqrt{-1}h_{i\overline{j}}f_{,0},\ f_{,0i}-f_{,i0}=A_{ij}f_,^{\ j},\ f_{,ij\overline k}-f_{,i\overline kj}=2\sqrt{-1}h_{j\overline k}f_{,i0}+R_{i\ j\overline k}^{\ l}f_{,l}.$$

Define $\Delta f:=\displaystyle\frac 1 2(f_{,i}^{\ \ i}+f_{,\ \ i}^{\ i})$ as the sub-Laplacian operator on $M$, then $\Delta f=\operatorname{Re}f_{,i}^{\ \ i}$, $f_{,i}^{\ \ i}=\Delta f+n\sqrt{-1}f_{,0}$. Denote $f_{,i}f_,^{\ i}$ as $|\nabla f|^2$.

In this article, $M$ is a closed, oriented, strictly pseudoconvex CR manifold of hypersurface type, and curvature and torsion satisfy the following pointwise condition:
\begin{equation}\label{condition}
  \operatorname{Ric}(Z,Z)\geqslant (n+1)\langle Z,Z\rangle_{L_\theta},~~\operatorname{Tor}(Z,Z)=0,~~\forall Z\in T^{(1,0)}M.
\end{equation}

Let $\alpha>1$, $\lambda>0$, $u\in C^\infty(M)$ is positive, we study the following equation:

\begin{equation}\label{equ}
  \Delta u-\lambda u+u^\alpha=0.
\end{equation}

',' would be omitted while writing derivatives of solution $u$ in this article.

Positive solutions don't exist if $\lambda\leqslant0$ by directly integrated by part, hence assume that $\lambda>0$. The existence and regularity of solutions are discussed carefully in \cite{MR0880182}, then only the classification of smooth solutions is considered in this article.

In celebrated paper \cite{MR0924699}, Jerison-Lee introduced remarkable identities to deduce the following theorem.

\begin{theorem}[\cite{MR0924699} Theorem A]\label{JL1}
    Assume that $u>0$ satisfying \eqref{equ} with $\alpha=\displaystyle\frac{n+2}{n}$, $\lambda=\displaystyle\frac{n^2}{4}$, and $(M^{2n+1},\theta)=(\mathbb S^{2n+1},\theta_c)$, then there exists $s\geqslant0$, $\xi\in\mathbb S^{2n+1}$ such that
    $$u(z)=c_{n,s}|\cosh s+(\sinh s)\langle z,\xi\rangle|^{-n},~z\in\mathbb S^{2n+1}.$$
\end{theorem}

For the flat CR manifold Heisenberg group $\mathbb H^n$ case, Jerison-Lee studied critical exponent $\alpha=\displaystyle\frac{n+2}{n}$ in \cite{MR0924699}, where they found three-dimensional family identities. The positive solution of the CR-Yamabe equation can be classified with the assumption of finite energy by using those identities:

\begin{theorem}[\cite{MR0924699} Corollary C]\label{JL1'}
  Assume that $u\in L^{\frac{2n+2}{n}}(\mathbb H^n)$ is the positive solution of $\Delta u+u^{\frac{n+2}{n}}=0$ in $\mathbb H^n$, then there exists $\lambda\in\mathbb C$ and $\mu\in\mathbb C^n$ satisfying $\operatorname{Im}\lambda>\displaystyle\frac{|\mu|^2}{4}$, such that
  $$u(z,t)=c_{n,\lambda,\mu}\left|t+\sqrt{-1}|z|^2+z\cdot\mu+\lambda\right|^{-n}.$$
\end{theorem}

Theorem \ref{JL1} is covered by Theorem \ref{JL1'} by Cayley transformation. The proofs of the above theorems were based on the same idea as Obata's \cite{MR0303464} proof 
of the analogous result in Riemannian geometry: the only Riemannian metrics 
on the sphere that are conformal to the standard one and have constant scalar 
curvature are obtained from the standard metric by a conformal diffeomorphism 
of the sphere. Since the 
pseudohermitian Bianchi identities involve extra torsion terms and on the Heisenberg group, this reflects the nontrivial commutation relations. Using computer algebra, Jerison-Lee found a three-dimensional family ((4.2)$\sim$(4.4) in \cite{MR0924699}) of solutions with divergence terms on the left-hand side and positive terms on the right-hand side. Then, the divergence theorem would prove that the right-hand side vanishes identically and gets above classification results. In page 4 of \cite{MR0924699}, Jerison-Lee raised the following problem:
\begin{quote}
  \emph{\indent An interesting (but vaguely defined) problem raised by this work is to find an "explanation" for the existence of divergence formulas such as (4.2) and (3.1). Is there a theoretical framework that would predict the existence and the structure of such formulas, so that they could be discovered more systematically?}
\end{quote}

With the help of dimensional conservation and invariant tensors, we state the following theorem, which answers the problem above.
 
\begin{theorem}\label{JL2}
  Assume that $u$ is the positive solution of $\Delta u+u^{\frac{n+2}{n}}=0$ in $\mathbb H^n$, then all the useful identities of $\{(0,0),2,6,+\}$ type must lie in the three-dimensional family.
\end{theorem}

The meaning of dimensional symbol $\{(0,0),2,6,+\}$ is given in Section 2. In \cite{MR3342188}, Wang extended Theorem \ref{JL1} from $(\mathbb S^{2n+1},\theta_c)$ to closed Einstein pseudohermitian manifold, and to closed pseudohermitian manifold under condition \eqref{condition} in \cite{MR4418321} where he raised the following conjecture.

\begin{conjecture}[\cite{MR4418321} Conjecture 1]\label{conjecture}
    If $1<\alpha\leqslant\displaystyle\frac{n+2}{n}$ and $0<\lambda\leqslant\displaystyle\frac{n}{2(\alpha-1)}$, the only positive solution of \eqref{equ} under the condition \eqref{condition} must be $u\equiv\displaystyle\lambda^{\frac{1}{\alpha-1}}$, otherwise $\alpha=\displaystyle\frac{n+2}{n}$, $\lambda=\displaystyle\frac{n^2}{4}$, $(M^{2n+1},\theta)=(\mathbb S^{2n+1},\theta_c)$ is standard CR sphere, with some $s\geqslant0$, $\xi\in\mathbb S^{2n+1}$ such that
    $$u(z)=c_{n,s}|\cosh s+(\sinh s)\langle z,\xi\rangle|^{-n},~z\in\mathbb S^{2n+1}.$$

\end{conjecture}

\begin{remark}
  In \cite{MR4418321}, Wang proved the above conjecture when $\alpha=\displaystyle\frac{n+2}{n}$ by using the Jerison-Lee identity (3.1) in \cite{MR0924699}.
\end{remark}

By dimensional conservation and invariant tensors, several new differential identities are obtained for the subcritical exponent case, which leads to our following main theorem. This theorem gives the positive answer to the subcritical exponent case of Conjecture \ref{conjecture}.

\begin{theorem}\label{CR}
    If $1<\alpha<\displaystyle\frac{n+2}{n}$ and $0<\lambda\leqslant\displaystyle\frac{n}{2(\alpha-1)}$, the only positive solution of \eqref{equ} under the condition \eqref{condition} must be $u\equiv\displaystyle\lambda^{\frac{1}{\alpha-1}}$.
\end{theorem}

 We let $u\in HW^{1,2}(M)$, which is the usual Sobolev space on $M$ (see \cite{MR0367477} or \cite{MR2214654} for details). In \cite{MR4418321}, the following Corollary \ref{Folland} was raised as the consequence of Theorem \ref{CR}.

\begin{corollary}\label{Folland}
    Let $(M^{2n+1},\theta)$ be a closed, oriented, strictly-pseudoconvex CR manifold of hypersurface type satisfying \eqref{condition}, then for $2<q<\displaystyle\frac{2Q}{Q-2}=\frac{2n+2}{n}$, $u\in HW^{1,2}(M)$, 
    $$\frac{4(q-2)}{Q-2}\int_M|\nabla u|^2+\int_M|u|^2\geqslant\operatorname{vol}(M)^{1-\frac 2 q}\left(\int_M|u|^q\right)^{\frac 2 q}.$$
    Equality holds if and only if $u$ is constant. Notice that $Q=2n+2$, and $q$ is equivalent to $\alpha+1$ in Theorem \ref{CR}.
\end{corollary}

The proof of Corollary \ref{Folland} is similar to the Riemannian case as corollary 6.2 in \cite{MR1134481}. It's noteworthy that Frank-Lieb got this inequality on standard CR sphere $(\mathbb S^{2n+1},\theta_c)$ by harmonic polynomials extension in \cite{MR2925386} (see corollary 2.3), and in fact, Frank-Lieb derived the sharp constants for Hardy-Littlewood-Sobolev inequalities on Heisenberg group. Besides, Moser-Trudinger and Beckner-Onofri's inequalities on the CR sphere are obtained by Branson-Fontana-Morpurgo in \cite{MR2999037}.

\begin{remark}
  Motivated by the Jerison-Lee identity (4.2) in \cite{MR0924699}, Ma-Ou \cite{MR4530625} proved that there is no positive solution of $\Delta u+u^{\alpha}=0$ in $\mathbb H^n$ while $1<\alpha<\displaystyle\frac{n+2}{n}$. Recently, Catino- Li-Monticelli-Roncoron  \cite{1} and Flynn-V\'{e}tois in \cite{2} got more generalizations for the critical exponent case.
\end{remark}

The result in Ma-Ou \cite{MR4530625} can be proved by the identity \eqref{case2} raised in this article with curvature terms discarded.

For semilinear elliptic equations on compact Riemannian manifolds, M-F. B. V\'eron and L. V\'eron [\cite{MR1134481}, Theorem 6.1]  introduced the Bochner-Lichnerowicz-Weitzenbeck formula in such a way
that they could extend and simplify  Gidas-Spruck's  \cite{MR0615628} results:

\begin{theorem}[\cite{MR1134481} Theorem 6.1]
  Assume that $(M^n,g)$ is a closed Riemannian manifold of dimension $n\geqslant 2$, $\alpha>1$, $\lambda>0$, and $u$ is a positive solution of 
  $$\Delta u-\lambda u+u^\alpha=0.$$ 
  Assume also that the spectrum $\sigma(R(x))$ of the Ricci tensor $R$ of the metric $g$ satisfies
  $$\inf\limits_{x\in M}\min \sigma(R(x)) \geqslant \frac{n-1}{n}(\alpha-1)\lambda,~~\alpha\leqslant\frac{n+2}{n-2}.$$
  Moreover, assume that one of the two inequalities is strict if $(M^n,g)=(\mathbb{S}^n,g_c)$ is the standard sphere, then $u\equiv{\lambda}^\frac{1}{\alpha-1}$.
\end{theorem}

The direct extension of Jerison-Lee identity fails in the process of solving the subcritical exponent case in CR geometry, because of the concurrence of curvature, torsion, and the second layer of the CR manifold as a Carnot group. Hence, dimensional conservation and invariant tensors are introduced in Section 2. Then, we get some new differential identities in Section 3 to prove Theorem \ref{CR}. Besides, the new method can explain the existence of such a three-dimensional family in the critical exponent case for the Heisenberg group. We discuss it in detail in Section 4 and prove Theorem \ref{JL2}, which answers the question of the theoretical framework for finding differential identities raised by Jerison-Lee \cite{MR0924699}.
\section{Preparations: dimensional conservation and invariant tensors}

In this section, \textbf{dimensional conservation} and \textbf{invariant tensors} are introduced for preparing useful differential identities. Target identities are composed of divergence of some vector fields and summation of positive terms which contain the complete square of some tensors, then all tensors in complete square terms are zero by divergence theorem. Thus, how to find those tensors priorly is essential.

We say a tensor $S(u)$ is of $\{(r,s),x,y,+/-\}$ type, if it's linearly composed of some $T^{(r,s)}$ tensors with $x$-degree $u$, $y$-order derivatives, and the number of $\sqrt{-1}$ plus the number of vector field $T$ is even/odd for every tensors. For example:
$$\{(2,0),1,2,+\}:~D_{ij}=u_{ij}+c_1\displaystyle\frac{u_iu_j}{u};$$
$$\{(1,1),1,2,+\}:~E_{i\overline j}=u_{i\overline j}+c_2\frac{u_iu_{\overline j}}{u}+c_3\Delta uh_{i\overline j}+c_4n\sqrt{-1}u_0h_{i\overline j}+c_5\frac{|\nabla u|^2}{u}h_{i\overline j}+c_6\lambda uh_{i\overline j},$$
where $\{c_l\}_{l=1}^6$ are constants. It's noteworthy that $u_0$ is of $\{(0,0),1,2,-\}$ type, and $\lambda$ is of $\{(0,0),0,2,+\}$ type. The type of tensors is additive when several types of tensors are multiplied together. The type of tensors must be conserved in differential identities. We call this phenomenon as \textbf{dimensional conservation}.

Recall the Riemannian case. From Obata \cite{MR0303464}, Gidas-Spurck \cite{MR0303464}, M-F. B. V\'eron and L. V\'eron \cite{MR1134481}, and especially Dolbeault-Esteban-Loss \cite{MR3229793}, we know that differential identities are found by multiplying $\Delta u$ in both sides of the equation, and using divergence theorem. Namely,
$$(\Delta u)^2=(\Delta uu_i)_,^{\ i}-(\Delta u)_,^{\ i}u_i=-(u^{ji}u_i)_{,j}+\sum_{i,j=1}^n|u_{ij}|^2+(\Delta uu_i)_,^{\ i},$$
then $u_{ij}$ becomes the main term of some target tensor hoped to be zero. Similar as Riemannian case, by multiplying equation \eqref{equ} with $\Delta u$ and divergence theorem,
\begin{align*}
  &(\Delta u)^2=(\Delta uu_i)_{,}^{\ i}-(\Delta u)_{,}^{\ i}u_i-n\sqrt{-1}u_0\Delta u\\
  =&-(u_{\ j}^{j}+n\sqrt{-1}u_0)_{,}^{\ i}u_i+(\Delta uu_i)_{,}^{\ i}-n\sqrt{-1}u_0\Delta u\\
  =&-(u^{ij}u_i)_{,j}+\sum_{i,j=1}^n|u_{ij}|^2-(n+2)\sqrt{-1}u_0^{\ i}u_i+(\Delta uu_i)_{,}^{\ i}-n\sqrt{-1}u_0\Delta u,
\end{align*}
then we can yield $\displaystyle\sum_{i,j=1}^n|u_{ij}|^2$ term, hence consider $u_{ij}$ as the main term of one of the target tensors. By dimensional conservation, we need a $\{(2,0),1,2,+\}$ type tensor, hence consider $D_{ij}$ defined as above.

Similarly, use the divergence theorem in another way:
\begin{align*}
  &(\Delta u)^2=(\Delta uu_i)_{,}^{\ i}-(\Delta u)_{,}^{\ i}u_i-n\sqrt{-1}u_0\Delta u\\
  =&-(u_{j}^{\ j}-n\sqrt{-1}u_0)_{,}^{\ i}u_i+(\Delta uu_i)_{,}^{\ i}-n\sqrt{-1}u_0\Delta u\\
  =&-(u_j^{\ i}u_i)_,^{\ j}+u_j^{\ i}u_i^{\ j}+n\sqrt{-1}u_0^{\ i}u_i+(\Delta uu_i)_{,}^{\ i}-n\sqrt{-1}u_0\Delta u\\
  =&-(u_j^{\ i}u_i)_,^{\ j}+\sum_{i,j=1}^n|u_{i\overline j}|^2+2\sqrt{-1}u_0u_i^{\ i}+n\sqrt{-1}u_0^{\ i}u_i+(\Delta uu_i)_{,}^{\ i}-n\sqrt{-1}u_0\Delta u,
\end{align*}
then $|u_{i\overline j}|^2$ term can be attained, hence consider a $\{(1,1),1,2,+\}$ type tensor $E_{i\overline j}$ defined as above. For $\displaystyle\sum_{i,j=1}^n|D_{ij}|^2$ and $\displaystyle\sum_{i,j=1}^n|E_{i\overline j}|^2$, $\{(0,0),2,4,+\}$ type identity is enough, such as Riemannian case. However, a $\{(1,0),1,3,-\}$ type tensor $G_i$ occurs by the following invariant tensors argument because of non-commutativity of $Z_i$ and $Z_{\overline i}$ caused by the second layer of CR manifold. \textbf{At last, we need a $\{(0,0),2,6,+\}$ type identity to deal with $\displaystyle\sum_{i=1}^n|G_i|^2$ term, which is just the key identity \eqref{2} in this paper.} The identity \eqref{2} has the same dimensional as Jerison-Lee's identities in \cite{MR0924699}.

Now, we hope that $D_{ij}$ and $E_{i\overline j}$ are zero for some $\alpha$ and $\lambda$. By $E_{i}^{\ i}=0$, we yield
$$c_3=-\frac 1 n,~~c_4=-\frac 1 n,~~c_5=-\frac 1 nc_2,~~c_6=0,$$
then $E_{i\overline j}=\displaystyle u_{i\overline j}+c_2\frac{u_iu_{\overline j}}{u}-\frac 1 n\left(\Delta u+n\sqrt{-1}u_0+c_2\frac{|\nabla u|^2}{u}\right)h_{i\overline j}$.

Set $D_i=\displaystyle\frac{D_{ij}u^j}{u}$, $E_i=\displaystyle\frac{E_{i\overline j}u^{\overline j}}{u}$. By direct computation and using equation \eqref{equ}:
\begin{align}
  \begin{split}\label{d'}
    D_{ij,}^{\ \ \ i}=&u_{ij}^{\ \ i}+c_1\frac{u_j^{\ i}u_i}{u}+c_1\frac{u_j(\Delta u+n\sqrt{-1}u_0)}{u}-c_1\frac{|\nabla u|^2}{u^2}u_j\\
    =&(\Delta u+n\sqrt{-1}u_0)_j+2\sqrt{-1}u_{0j}+R_{j\overline i}u^{\overline i}+c_1\Big[E_{j\overline i}-c_2\frac{u_ju_{\overline i}}{u}+\frac 1 n\Big(\Delta u\\
    &+n\sqrt{-1}u_0+c_2\frac{|\nabla u|^2}{u}\Big)h_{j\overline i}\Big]\frac{u^{\overline i}}{u}+c_1\frac{u_j(\Delta u+n\sqrt{-1}u_0)}{u}-c_1\frac{|\nabla u|^2}{u^2}u_j\\
    =&c_1E_j+(n+2)\sqrt{-1}u_{0j}+(n+1)c_1\frac{\sqrt{-1}u_0u_j}{u}+(\frac{n+1}{n}c_1+\alpha)\frac{\Delta u}{u}u_j\\
    &-(\frac{n-1}{n}c_2+1)c_1\frac{|\nabla u|^2}{u^2}u_j+R_{j\overline i}u^{\overline i}+(1-\alpha)\lambda u_j,
  \end{split}
\end{align}
and
\begin{align}
  \begin{split}\label{e'}
    E_{i\overline j,}^{\ \ \ i}=&u_{i\overline j}^{\ \ i}+c_2\frac{u_{\overline j}^{\ i}u_i}{u}+c_2\frac{u_{\overline j}(\Delta u+n\sqrt{-1}u_0)}{u}-c_2\frac{|\nabla u|^2}{u^2}u_{\overline j}\\
    &-\frac{(\Delta u+n\sqrt{-1}u_0)_{\overline j}}{n}-\frac{c_2}{n}\frac{u_{\overline i\overline j}u^{\overline i}}{u}-\frac{c_2}{n}\frac{u_{i\overline j}u^{i}}{u}+\frac{c_2}{n}\frac{|\nabla u|^2}{u^2}u_{\overline j}\\
    =&(\Delta u+n\sqrt{-1}u_0)_{\overline j}+\frac{n-1}{n}c_2\frac{u_{\overline i\overline j}u^{\overline i}}{u}+c_2\frac{u_{\overline j}(\Delta u+n\sqrt{-1}u_0)}{u}\\
    &-\frac{c_2}{n}\frac{u_{i\overline j}u^i}{u}-\frac{(\Delta u+n\sqrt{-1}u_0)_{\overline j}}{n}-\frac{n-1}{n}c_2\frac{|\nabla u|^2}{u^2}u_{\overline j}\\
    =&-\frac{c_2}{n}\left[E_{i\overline j}-c_2\frac{u_iu_{\overline j}}{u}+\frac 1 n\left(\Delta u+n\sqrt{-1}u_0+c_2\frac{|\nabla u|^2}{u}\right)h_{i\overline j}\right]\frac{u^i}{u}\\
    &+\frac{n-1}{n}c_2\left(D_{\overline i\overline j}-c_1\frac{u_{\overline i}u_{\overline j}}{u}\right)\frac{u^{\overline i}}{u}+(n-1)\sqrt{-1}u_{0\overline j}+nc_2\frac{\sqrt{-1}u_0u_{\overline j}}{u}\\
    &+(c_2+\frac{n-1}{n}\alpha)\frac{\Delta u}{u}u_{\overline j}-\frac{n-1}{n}c_2\frac{|\nabla u|^2}{u^2}u_{\overline j}+\frac{n-1}{n}(1-\alpha)\lambda u_{\overline j}\\
    =&\frac{n-1}{n}c_2D_{\overline j}-\frac{c_2}{n}E_{\overline j}+(n-1)\sqrt{-1}u_{0\overline j}+\frac{n^2-1}{n}c_2\frac{\sqrt{-1}u_0u_{\overline j}}{u}+\frac{n-1}{n}\times\\
    &(\frac{n+1}{n}c_2+\alpha)\frac{\Delta u}{u}u_{\overline j}-\frac{n-1}{n}(c_1-\frac{c_2}{n}+1)c_2\frac{|\nabla u|^2}{u^2}u_{\overline j}+\frac{n-1}{n}(1-\alpha)\lambda u_{\overline j}.
  \end{split}
\end{align}

If $D_{ij}$ and $E_{i\overline j}$ are 0, then $D_{ij,}^{\ \ \ i}$ and $E_{i\overline j,}^{\ \ \ i}$ are also 0, hence 
\begin{align}
  \begin{split}\label{proportional1}
    0=&(n+2)\sqrt{-1}u_{0j}+(n+1)c_1\frac{\sqrt{-1}u_0u_j}{u}+(\frac{n+1}{n}c_1+\alpha)\frac{\Delta u}{u}u_j\\
    &-(\frac{n-1}{n}c_2+1)c_1\frac{|\nabla u|^2}{u^2}u_j+R_{j\overline i}u^{\overline i}+(1-\alpha)\lambda u_j,
  \end{split}
\end{align}
and
\begin{align}
  \begin{split}\label{proportional2}
    0=&(n-1)\sqrt{-1}u_{0\overline j}+\frac{n^2-1}{n}c_2\frac{\sqrt{-1}u_0u_{\overline j}}{u}+\frac{n-1}{n}(\frac{n+1}{n}c_2+\alpha)\frac{\Delta u}{u}u_{\overline j}\\
    &-\frac{n-1}{n}(c_1-\frac{c_2}{n}+1)c_2\frac{|\nabla u|^2}{u^2}u_{\overline j}+\frac{n-1}{n}(1-\alpha)\lambda u_{\overline j}.
  \end{split}
\end{align}
Let the coefficients of $\sqrt{-1}u_{0j}$, $\displaystyle\frac{\sqrt{-1}u_0u_j}{u}$, $\displaystyle\frac{\Delta u}{u}u_j$ and $\displaystyle\frac{|\nabla u|^2}{u^2}u_j$ in \eqref{proportional1} and \eqref{proportional2} are proportional:
$$\frac{n+2}{-(n-1)}=\frac{(n+1)c_1}{\displaystyle-\frac{n^2-1}{n}c_2}=\frac{\displaystyle\frac{n+1}{n}c_1+\alpha}{\displaystyle\frac{n-1}{n}(\frac{n+1}{n}c_2+\alpha)}=\frac{-\displaystyle(\frac{n-1}{n}c_2+1)c_1}{\displaystyle-\frac{n-1}{n}(c_1-\frac{c_2}{n}+1)c_2},$$
then $c_1=c_2=\alpha=0$ or $c_1=-\displaystyle\frac{n+2}{n}$, $c_2=-1$, $\alpha=\displaystyle\frac{n+2}{n}$. \textbf{Hence the critical exponent $\alpha=\displaystyle\frac{n+2}{n}$ can be determined with this method. $c_1=-\displaystyle\frac{n+2}{n}$ and $c_2=-1$ in the critical exponent case are essential when we answer the question raised by Jerison-Lee and prove the Theorem \ref{JL2} in Section 4.}

In the following, we concentrate on the subcritical exponent case $1<\alpha<\displaystyle\frac{n+2}{n}$. \textbf{By the rigidity theorem in Riemannian case, such as \cite{MR3229793}, the terms with $|\nabla u|^4$ are needed in the proof of Theorem \ref{CR}.}

Hence, only let the coefficients of $\sqrt{-1}u_{0i}$, $\displaystyle\frac{\sqrt{-1}u_0u_i}{u}$ and $\displaystyle\frac{\Delta u}{u}u_i$ in \eqref{proportional1} and \eqref{proportional2} are proportional:
$$\frac{n+2}{-(n-1)}=\frac{(n+1)c_1}{\displaystyle-\frac{n^2-1}{n}c_2}=\frac{\displaystyle\frac{n+1}{n}c_1+\alpha}{\displaystyle\frac{n-1}{n}(\frac{n+1}{n}c_2+\alpha)},$$
then $c_1=-\alpha$, $c_2=-\displaystyle\frac{n\alpha}{n+2}$. Rewrite $D_{ij}$, $E_{i\overline j}$, and define a $\{(1,0),1,3,+\}$ type tensor $G_i$:

$$D_{ij}=u_{ij}-\alpha\frac{u_iu_j}{u},$$
$$E_{i\overline j}=u_{i\overline j}-\frac{n\alpha}{n+2}\frac{u_iu_{\overline j}}{u}-\frac 1 n\left(\Delta u+n\sqrt{-1}u_0-\frac{n\alpha}{n+2}\frac{|\nabla u|^2}{u}\right)h_{i\overline j},$$
\begin{align*}
  G_i=&n\sqrt{-1}u_{0i}-\frac{n(n+1)}{n+2}\alpha\frac{\sqrt{-1}u_0u_i}{u}-\frac{\alpha}{n+2}\frac{\Delta u}{u}u_i\\
  &+\frac{n\alpha}{n+2}\left(\frac{n+1}{n+2}\alpha-1\right)\frac{|\nabla u|^2}{u^2}u_i+(\alpha-1)\lambda u_i.
\end{align*}
Rewrite $D_{ij,}^{\ \ \ i}$ and $E_{i\overline j,}^{\ \ \ i}$ in \eqref{d'} and \eqref{e'}:
\begin{equation}\label{Dij}
  D_{ij,}^{\ \ \ i}=-\alpha E_j+\frac{n+2}{n}G_j+2\alpha(1-\frac{n\alpha}{n+2})\frac{|\nabla u|^2}{u^2}u_j+R_{j\overline i}u^{\overline i}-\frac{2(n+1)}{n}(\alpha-1)\lambda u_j,
\end{equation}
\begin{equation}\label{Eij}
  E_{i\overline j,}^{\ \ \ i}=-\frac{n-1}{n+2}\alpha D_{\overline j}+\frac{\alpha}{n+2}E_{\overline j}-\frac{n-1}{n}G_{\overline j}.
\end{equation}
Now, the covariant derivatives of $D_{ij}$ and $E_{i\overline j}$ and $G_i$ are also composed of $D_{ij}$, $E_{i\overline j}$, $G_i$ and $\displaystyle\frac{|\nabla u|^2}{u^2}u_j$ terms in some suitable curvature condition. The $\displaystyle\frac{|\nabla u|^2}{u^2}u_j$ term is vanishing in the critical exponent case. The invariance of $D_{ij}$, $E_{i\overline j}$, and $G_i$ in differentiating process are reasonable since those tensors are hoped to be zero. Hence, we call $D_{ij}$, $E_{i\overline j}$ and $G_i$ as \textbf{invariant tensors}. With the invariance arguments above, invariant tensors can be deduced without any geometric background.

Some notations are needed:
$$E_{\overline ij}=\overline{E_{i\overline j}},~~L_{i\overline j}=\frac{u_iu_{\overline j}}{u}-\frac 1 n\frac{|\nabla u|^2}{u}h_{i\overline j},~~\mathscr R=R_{i\overline j}u^iu^{\overline j}-\frac{2(n+1)}{n}(\alpha-1)\lambda|\nabla u|^2.$$

For convenience, the following four lemmas are needed.

\begin{lemma}\label{prep1}
  \ 

  (1) $E_i^{\ i}:=E_{i\overline j}h^{i\overline j}=0$, $E_{i\overline j}=E_{\overline ji}$, $E_iu^i\in\mathbb R$;

  (2) $L_{i}^{\ i}=0$, $E_iu^i=E_{i\overline j}L^{i\overline j}$, $\displaystyle\sum_{i,j=1}^n|L_{i\overline j}|^2=\frac{n-1}{n}\frac{|\nabla u|^4}{u^2}$;

  (3) Assume that \eqref{condition} and $\lambda\leqslant\displaystyle\frac{n}{2(\alpha-1)}$ hold, then $\mathscr R\geqslant0$.
\end{lemma}

\begin{proof}
  $\displaystyle E_i^{\ i}=u_i^{\ i}-\frac{n\alpha}{n+2}\frac{|\nabla u|^2}{u}-\left(\Delta u+n\sqrt{-1}u_0-\frac{n\alpha}{n+2}\frac{|\nabla u|^2}{u}\right)=0$, then $L_i^{\ i}=0$ is proved similarly. Hence, $\displaystyle E_{i\overline j}L^{i\overline j}=E_{i\overline j}\cdot\frac{u^iu^{\overline j}}{u}=E_iu^i$, $\displaystyle\sum_{i,j=1}^n|L_{i\overline j}|^2=L_{i\overline j}\cdot\frac{u^{i}u^{\overline j}}{u}=\frac{n-1}{n}\frac{|\nabla u|^4}{u^2}$.

  By $u_{i\overline j}-u_{\overline ji}=2\sqrt{-1}h_{i\overline j}u_0$, we yield that
  \begin{align*}
    E_{i\overline j}=&\displaystyle u_{i\overline j}-\frac{n\alpha}{n+2}\frac{u_iu_{\overline j}}{u}-\frac 1 n\left(\Delta u+n\sqrt{-1}u_0-\frac{n\alpha}{n+2}\frac{|\nabla u|^2}{u}\right)h_{i\overline j}\\
    =&\displaystyle u_{\overline j i}-\frac{n\alpha}{n+2}\frac{u_{\overline j}u_i}{u}-\frac 1 n\left(\Delta u-n\sqrt{-1}u_0-\frac{n\alpha}{n+2}\frac{|\nabla u|^2}{u}\right)h_{\overline ji}=E_{\overline ji},
  \end{align*}
  then $E_iu^i=\displaystyle\frac{E_{i\overline j}u^{i}u^{\overline j}}{u}=\frac{E_{\overline ji}u^{i}u^{\overline j}}{u}=E_{\overline j}u^{\overline j}$, i.e. $E_iu^i\in\mathbb R$.

  If $\operatorname{Ric}(Z,Z)\geqslant(n+1)\langle Z,Z\rangle_{L_\theta}$ and $\lambda\leqslant\displaystyle\frac{n}{2(\alpha-1)}$,
  $$\mathscr R\geqslant (n+1)|\nabla u|^2-\frac{2(n+1)}{n}(\alpha-1)\lambda|\nabla u|^2\geqslant0.$$
\end{proof}

\begin{lemma}\label{Cauchy}
    $\displaystyle\frac{|\nabla u|^2}{u^2}\sum_{i,j}|D_{ij}|^2\geqslant\sum_i|D_i|^2$, $\displaystyle\frac{|\nabla u|^2}{u^2}\sum_{i,j}|E_{i\overline j}|^2\geqslant\frac{n}{n-1}\sum_i|E_i|^2$ if $n\geqslant2$.
\end{lemma}
\begin{proof}
    Assume that $A\in\mathbb C^{n\times n}$ is Hermitian, $\mu\in\mathbb C^{n\times 1}$. By Cauchy inequality, $$\sum_{j=1}^n|A_{ij}\mu_j|^2\leqslant\sum_{j=1}^n|A_{ij}|^2\|\mu\|^2.$$
    Sum $i$ from 1 to $n$: $\displaystyle\sum_{i,j=1}^n|A_{ij}\mu_j|^2\leqslant\sum_{i,j=1}^n|A_{ij}|^2\|\mu\|^2$. Then $u^2\displaystyle\sum_i|D_i|^2\leqslant|\nabla u|^2\sum_{i,j}|D_{ij}|^2$.

    For $n\geqslant2$, assume that $\operatorname{tr}A=0$ additionally. Without loss of generality, assume that $A_{ij}=0$ if $i\neq j$ and $i,j\geqslant2$, $\mu=(1,0,\cdots,0)^T$, then
    \begin{align*}
        &\sum_{i,j=1}^n|A_{ij}|^2\|\mu\|^2-\frac{n}{n-1}\sum_{i,j=1}^n|A_{ij}\mu_j|^2\\
        &=\sum_{i=1}^n|A_{ii}|^2+2\sum_{i=2}^n|A_{i1}|^2-\frac{n}{n-1}|A_{11}|^2-\frac{n}{n-1}\sum_{i=2}^n|A_{i1}|^2\\
        &\geqslant\sum_{i=2}^n|A_{ii}|^2-\frac{1}{n-1}|A_{11}|^2\\
        &\xlongequal{\operatorname{tr}A=0}\frac{1}{n-1}\sum_{2\leqslant i<j\leqslant n}|A_{ii}-A_{jj}|^2\geqslant0.
    \end{align*}
    Hence $|\nabla u|^2\displaystyle\sum_{i,j}|E_{i\overline j}|^2\geqslant\frac{n}{n-1}u^2\sum_i|E_i|^2$ for $n\geqslant 2$.
\end{proof}

\begin{lemma}\label{prep2}
  $$(\Delta u)_{,i}=\lambda u_i-\alpha u^{\alpha-1}u_i=\alpha\frac{\Delta u}{u}u_i+(1-\alpha)\lambda u_i,$$
  $$(|\nabla u|^2)_{,\overline i}=uD_{\overline i}+uE_{\overline i}+\frac{2n+1}{n+2}\alpha\frac{|\nabla u|^2}{u}u_{\overline i}+\frac 1 n\Delta u u_{\overline i}+\sqrt{-1}u_0u_{\overline i},$$
  \begin{align*}
    n\sqrt{-1}u_{0\overline i}=&-G_{\overline i}+\frac{n(n+1)}{n+2}\alpha\frac{\sqrt{-1}u_0u_{\overline i}}{u}-\frac{\alpha}{n+2}\frac{\Delta u}{u}u_{\overline i}\\
    &+\frac{n\alpha}{n+2}\left(\frac{n+1}{n+2}\alpha-1\right)\frac{|\nabla u|^2}{u^2}u_{\overline i}+(\alpha-1)\lambda u_{\overline i}.
  \end{align*}
\end{lemma}

\begin{proof}
  They can be checked by equation \eqref{equ} and definitions of $D_{ij}$, $E_{i\overline j}$ and $G_i$ easily.
\end{proof}

\begin{lemma}\label{invariance}
  \begin{align}\label{d}
    \begin{split}
      D_{i,}^{\ \ i}=&u^{-1}\sum_{i,j=1}^n|D_{ij}|^2+(\alpha-1)\frac{D_iu^i}{u}-\alpha\frac{E_iu^i}{u}+\frac{n+2}{n}\frac{G_iu^i}{u}\\
      &+2\alpha(1-\frac{n\alpha}{n+2})\frac{|\nabla u|^4}{u^3}+u^{-1}\mathscr R,
    \end{split}
  \end{align}
  \begin{equation}\label{e}
    E_{i,}^{\ \ i}=u^{-1}\sum_{i,j=1}^n|E_{i\overline j}|^2-\frac{n-1}{n+2}\alpha\frac{D_{\overline i}u^{\overline i}}{u}+\left(\frac{n+1}{n+2}\alpha-1\right)\frac{E_{\overline i}u^{\overline i}}{u}-\frac{n-1}{n}\frac{G_{\overline i}u^{\overline i}}{u},
  \end{equation}
  \begin{equation}\label{g}
      \operatorname{Im}G_{i,}^{\ \ i}=\operatorname{Im}\left[\frac{n\alpha}{n+2}\left(\frac{n+1}{n+2}\alpha-1\right)\frac{D_{\overline i}u^{\overline i}}{u}+\frac{n+1}{n+2}\alpha\frac{G_{\overline i}u^{\overline i}}{u}\right].
  \end{equation}
\end{lemma}

\begin{proof}
  \eqref{d} and \eqref{e} can be checked directly by \eqref{Dij} and \eqref{Eij}. By Lemma \ref{prep2},
  \begin{align*}
    &\operatorname{Im}G_{i,}^{\ \ i}\\
    =&n\operatorname{Im}\sqrt{-1}(\Delta u)_{,0}-\frac{n(n+1)}{n+2}\alpha\left(\operatorname{Im}\frac{\sqrt{-1}u_{0\overline i}u^{\overline i}}{u}+\frac{u_0\Delta u}{u}-\frac{u_0|\nabla u|^2}{u^2}\right)\\
    &-\frac{n\alpha}{n+2}\frac{u_0\Delta u}{u}+\frac{n\alpha}{n+2}\left(\frac{n+1}{n+2}\alpha-1\right)\left(\operatorname{Im}\frac{D_{\overline i}u^{\overline i}}{u}+(n+1)\frac{u_0|\nabla u|^2}{u^2}\right)+n(\alpha-1)\lambda u_0\\
    =&\frac{n\alpha}{n+2}\left(\frac{n+1}{n+2}\alpha-1\right)\operatorname{Im}\frac{D_{\overline i}u^{\overline i}}{u}-\frac{n(n+1)}{n+2}\alpha\operatorname{Im}\frac{\sqrt{-1}u_{0\overline i}u^{\overline i}}{u}+\frac{n(n+1)^2}{(n+2)^2}\alpha^2\frac{u_0|\nabla u|^2}{u^2}\\
    =&\operatorname{Im}\left[\frac{n\alpha}{n+2}\left(\frac{n+1}{n+2}\alpha-1\right)\frac{D_{\overline i}u^{\overline i}}{u}+\frac{n+1}{n+2}\alpha\frac{G_{\overline i}u^{\overline i}}{u}\right].
  \end{align*}
\end{proof}

\textbf{By invariance argument above, we need an identity including $\displaystyle\sum_{i,j}|D_{ij}|^2$, $\displaystyle\sum_{i,j}|E_{i\overline j}|^2$ and $\displaystyle\sum_i|G_i|^2$. Because of $\displaystyle\sum_i|G_i|^2$, $\{(0,0),2,4,+\}$ type identity is not enough. Hence consider the following $\{(0,0),2,6,+\}$ type identity, which is crucial for the proof of Theorem \ref{CR}.}
\begin{proposition}
  Let $\{d_l\}_{l=1}^4$, $\{e_l\}_{l=1}^4$, $\mu$ and $\beta$ be undetermined constants, then
  \begin{align}
    \begin{split}\label{2}
      &u^{-\beta}\operatorname{Re}\Big\{u^\beta\Big[\left(d_1\frac{|\nabla u|^2}{u}+d_2 u^\alpha+d_3\lambda u+d_4n\sqrt{-1}u_0\right)D_i\\
      &+\left(e_1\frac{|\nabla u|^2}{u}+e_2 u^\alpha+e_3\lambda u+e_4n\sqrt{-1}u_0\right)E_i-\mu n\sqrt{-1}u_0G_i\Big]\Big\}_,^{\ i}\\
      =&\left[d_1\frac{|\nabla u|^2}{u^2}+d_2u^{\alpha-1}+d_3\lambda\right]\left[\sum_{i,j}|D_{ij}|^2+2\alpha(1-\frac{n\alpha}{n+2})\frac{|\nabla u|^4}{u^2}+\mathscr R\right]\\
      &+\left[e_1\frac{|\nabla u|^2}{u^2}+e_2u^{\alpha-1}+e_3\lambda\right]\sum_{i,j}|E_{i\overline j}|^2+d_1\sum_i|D_i|^2+e_1\sum_i|E_i|^2\\
     &+\mu\sum_i|G_i|^2+(d_1+e_1)\operatorname{Re}D_iE^i-d_4\operatorname{Re}D_iG^i-e_4\operatorname{Re}E_iG^i\\
     &+\operatorname{Re}\left[\Delta_1\frac{|\nabla u|^2}{u^2}+\Delta_2u^{\alpha-1}+\Delta_3\lambda+\Delta_4\frac{n\sqrt{-1}u_0}{u}\right]D_iu^i\\
     &+\left[\Theta_1\frac{|\nabla u|^2}{u^2}+\Theta_2u^{\alpha-1}+\Theta_3\lambda\right]E_iu^i\\
     &+\operatorname{Re}\left[\Xi_1\frac{|\nabla u|^2}{u^2}+\Xi_2u^{\alpha-1}+\Xi_3\lambda+\Xi_4\frac{n\sqrt{-1}u_0}{u}\right]G_iu^i.
    \end{split}
  \end{align}
  The coefficients are:
  $$\Delta_1=\left(\beta+\frac{3(n+1)}{n+2}\alpha-2\right)d_1-\frac{n-1}{n+2}\alpha e_1+\frac{n\alpha}{n+2}(\frac{n+1}{n+2}\alpha-1)d_4,$$
  $$\Delta_2=-\frac 1 nd_1+(\beta+2\alpha-1)d_2-\frac{n-1}{n+2}\alpha e_2+\frac{\alpha}{n+2}d_4,$$
  $$\Delta_3=\frac 1 n d_1+(\beta+\alpha)d_3-\frac{n-1}{n+2}\alpha e_3+(\frac{n+1}{n+2}\alpha-1)d_4,$$
  $$\Delta_4=\frac 1 nd_1+\left(\beta+\frac{2n+3}{n+2}\alpha-1\right)d_4+\frac{n-1}{n+2}\alpha e_4+\frac{n\alpha}{n+2}(\frac{n+1}{n+2}\alpha-1)\mu,$$
  $$\Theta_1=-\alpha d_1+\left(\beta+\frac{3n+2}{n+2}\alpha-2\right)e_1+\frac{n\alpha}{n+2}(\frac{n+1}{n+2}\alpha-1)e_4,$$
  $$\Theta_2=-\frac 1 ne_1-\alpha d_2+
  \left(\beta+\frac{2n+3}{n+2}\alpha-1\right)e_2+\frac{\alpha}{n+2}e_4,$$
  $$\Theta_3=\frac 1 ne_1-\alpha d_3+\left(\beta+\frac{n+1}{n+2}\alpha\right) e_3+(\frac{n+1}{n+2}\alpha-1)e_4,$$
  $$\Xi_1=\frac{n+2}{n}d_1-\frac{n-1}{n}e_1-\frac{n\alpha}{n+2}(\frac{n+1}{n+2}\alpha-1)\mu,$$
  $$\Xi_2=\frac{n+2}{n}d_2-\frac{n-1}{n}e_2-\frac{\alpha}{n+2}\mu,$$
  $$\Xi_3=\frac{n+2}{n}d_3-\frac{n-1}{n}e_3-(\frac{n+1}{n+2}\alpha-1)\mu,$$
  $$\Xi_4=\frac{n+2}{n}d_4+\frac{n-1}{n}e_4-\beta\mu.$$
\end{proposition}

\begin{proof}
  By Lemma \ref{prep2} and Lemma \ref{invariance}, we yield the following $\{(0,0),2,6,+\}$ type identity:
  \begin{align*}
    &u^{-\beta}\operatorname{Re}(u^{\beta-1}|\nabla u|^2D_i)_,^{\ i}\\
    =&\frac{|\nabla u|^2}{u^2}\sum_{i,j}\left[|D_{ij}|^2+2\alpha(1-\frac{n\alpha}{n+2})\frac{|\nabla u|^4}{u^2}+\mathscr R\right]+\sum_i|D_i|^2+\operatorname{Re}D_iE^i\\
    &+\operatorname{Re}\left[\left(\beta+\frac{3(n+1)}{n+2}\alpha-2\right)\frac{|\nabla u|^2}{u^2}-\frac 1 n u^{\alpha-1}+\frac 1 n\lambda+\frac 1 n\frac{n\sqrt{-1}u_0}{u}\right]D_iu^i\\
    &-\alpha\frac{|\nabla u|^2}{u^2}E_iu^i+\frac{n+2}{n}\frac{|\nabla u|^2}{u^2}\operatorname{Re}G_iu^i,
  \end{align*}
  \begin{align*}
    &u^{-\beta}\operatorname{Re}(u^{\beta+\alpha}D_i)_,^{\ i}\\
    =&u^{\alpha-1}\sum_{i,j}\left[|D_{ij}|^2+2\alpha(1-\frac{n\alpha}{n+2})\frac{|\nabla u|^4}{u^2}+\mathscr R\right]\\
    &+(\beta+2\alpha-1)u^{\alpha-1}\operatorname{Re}D_iu^i-\alpha u^{\alpha-1}E_iu^i+\frac{n+2}{n}u^{\alpha-1}\operatorname{Re}G_iu^i,
  \end{align*}
  \begin{align*}
    &u^{-\beta}\operatorname{Re}(u^{\beta+1}\cdot\lambda D_i)_,^{\ i}\\
    =&\lambda\sum_{i,j}\left[|D_{ij}|^2+2\alpha(1-\frac{n\alpha}{n+2})\frac{|\nabla u|^4}{u^2}+\mathscr R\right]\\
    &+(\beta+\alpha)\lambda\operatorname{Re}D_iu^i-\alpha\lambda E_iu^i+\frac{n+2}{n}\lambda\operatorname{Re}G_iu^i,
  \end{align*}
  \begin{align*}
    &u^{-\beta}\operatorname{Re}(u^\beta\cdot n\sqrt{-1}u_0D_i)_,^{\ i}\\
    =&-\operatorname{Re}D_iG^i+\left(\beta+\frac{2n+3}{n+2}\alpha-1\right)\operatorname{Re}\frac{n\sqrt{-1}u_0}{u}D_iu^i+\frac{n+2}{n}\operatorname{Re}\frac{n\sqrt{-1}u_0}{u}G_iu^i\\
    &+\left[\frac{n\alpha}{n+2}\left(\frac{n+1}{n+2}\alpha-1\right)\frac{|\nabla u|^2}{u^2}+\frac{\alpha}{n+2}u^{\alpha-1}+\left(\frac{n+1}{n+2}\alpha-1\right)\lambda\right]\operatorname{Re}D_iu^i,
  \end{align*}
  \begin{align*}
    &u^{-\beta}\operatorname{Re}(u^{\beta-1}|\nabla u|^2E_i)_,^{\ i}\\
    =&\frac{|\nabla u|^2}{u^2}\sum_{i,j}|E_{i\overline j}|^2+\sum_i|E_i|^2+\operatorname{Re}D_iE^i-\frac{n-1}{n+2}\alpha\frac{|\nabla u|^2}{u^2}\operatorname{Re}D_iu^i\\
    &+\left[\left(\beta+\frac{3n+2}{n+2}\alpha-2\right)\frac{|\nabla u|^2}{u^2}-\frac 1 nu^{\alpha-1}+\frac 1 n\lambda\right]E_iu^i-\frac{n-1}{n}\frac{|\nabla u|^2}{u^2}\operatorname{Re}G_iu^i,
  \end{align*}
  \begin{align*}
    &u^{-\beta}\operatorname{Re}(u^{\beta+\alpha}E_i)_,^{\ i}\\
    =&u^{\alpha-1}\sum_{i,j}|E_{i\overline j}|^2-\frac{n-1}{n+2}\alpha u^{\alpha-1}\operatorname{Re}D_iu^i\\
    &+\left(\beta+\frac{2n+3}{n+2}\alpha-1\right)u^{\alpha-1}E_iu^i-\frac{n-1}{n}u^{\alpha-1}\operatorname{Re}G_iu^i,
  \end{align*}
  \begin{align*}
    &u^{-\beta}\operatorname{Re}(u^{\beta+1}\cdot\lambda E_i)_,^{\ i}\\
    =&\lambda\sum_{i,j}|E_{i\overline j}|^2-\frac{n-1}{n+2}\alpha\lambda\operatorname{Re}D_iu^i+\left(\beta+\frac{n+1}{n+2}\alpha\right)\lambda E_iu^i-\frac{n-1}{n}\lambda\operatorname{Re}G_iu^i.
  \end{align*}
  \begin{align*}
    &u^{-\beta}\operatorname{Re}(u^\beta\cdot n\sqrt{-1} u_0E_i),^{\ i}\\
    =&-\operatorname{Re}E_iG^i+\frac{n-1}{n+2}\alpha\operatorname{Re}\frac{n\sqrt{-1}u_0}{u}D_iu^i+\frac{n-1}{n}\operatorname{Re}\frac{n\sqrt{-1}u_0}{u}G_iu^i\\
    &+\left[\frac{n\alpha}{n+2}\left(\frac{n+1}{n+2}\alpha-1\right)\frac{|\nabla u|^2}{u^2}+\frac{\alpha}{n+2}u^{\alpha-1}+\left(\frac{n+1}{n+2}\alpha-1\right)\lambda\right]E_iu^i,
  \end{align*}
  \begin{align*}
    &u^{-\beta}\operatorname{Re}(-n\sqrt{-1}u^\beta u_0G_i)_,^{\ i}\\
    =&\sum_i|G_i|^2+\frac{n\alpha}{n+2}\left(\frac{n+1}{n+2}\alpha-1\right)\operatorname{Re}\frac{n\sqrt{-1}u_0}{u}D_iu^i-\beta \operatorname{Re}\frac{n\sqrt{-1}u_0}{u}G_iu^i\\
    &-\left[\frac{n\alpha}{n+2}\left(\frac{n+1}{n+2}\alpha-1\right)\frac{|\nabla u|^2}{u^2}+\frac{\alpha}{n+2}u^{\alpha-1}+\left(\frac{n+1}{n+2}\alpha-1\right)\lambda\right]\operatorname{Re}G_iu^i.
  \end{align*}
  Then identity \eqref{2} can be proved by linearly combining them.
\end{proof}

When $n=1$, because $E_{1\overline 1}=0$, we need some other vector fields, whose divergence is composed of invariant tensors as well. The following identity satisfies our demand.
\begin{proposition}
  If $n=1$, let $\beta$ be an undetermined constant, then
  \begin{align}
    \begin{split}\label{3}
      &u^{-\beta}\operatorname{Re}\Big\{u^\beta\Big[\left(\frac{\alpha}{3}(\frac 1 2-\frac \alpha 3)\frac{|\nabla u|^2}{u^2}-\frac\alpha 6u^{\alpha-1}+(\frac1 2-\frac \alpha 3)\lambda\right)\frac{|\nabla u|^2}{u}\\
      &+\left(\frac 1 2(\beta+\frac 4 3\alpha-1)\frac{|\nabla u|^2}{u^2}-u^{\alpha-1}+\lambda-\frac{\sqrt{-1}u_0}{u}\right)\sqrt{-1}u_0\Big]u_1\Big\}_,^{\ 1}\\
      =&\operatorname{Re}\left[\frac\alpha 3(1-\frac 2 3\alpha)\frac{|\nabla u|^2}{u^2}-\frac\alpha 6u^{\alpha-1}+(\frac 1 2-\frac\alpha 3)\lambda-\frac 1 2(\beta+\frac 4 3\alpha-1)\frac{\sqrt{-1}u_0}{u}\right]D_1u^1\\
      &+\operatorname{Re}\left[-\frac 1 2(\beta+\frac 4 3\alpha-1)\frac{|\nabla u|^2}{u^2}+u^{\alpha-1}-\lambda-2\frac{\sqrt{-1}u_0}{u}\right]G_1u^1\\
      &-\frac\alpha 3(1-\frac\alpha 3)(1-\frac 2 3\alpha)\frac{|\nabla u|^6}{u^4}.
    \end{split}
  \end{align}
\end{proposition}

\begin{proof}
  Notice that $E_{1\overline 1}=0$ in the case $n=1$. By Lemma \ref{prep2}, we yield the following $\{(0,0),2,6,+\}$ type identities:
  \begin{equation}\label{A1}
    u^{-\beta}\operatorname{Re}(u^{\beta-3}|\nabla u|^4u_1),^{\ 1}=2\frac{|\nabla u|^2}{u^2}\operatorname{Re}D_1u^1+(\beta+2\alpha-3)\frac{|\nabla u|^6}{u^4}+3\Delta u\frac{|\nabla u|^4}{u^3},
  \end{equation}
  $$u^{-\beta}\operatorname{Re}(u^{\beta+\alpha-2}|\nabla u|^2u_1),^{\ 1}=u^{\alpha-1}\operatorname{Re}D_1u^1+(\beta+2\alpha-2)u^{\alpha-3}|\nabla u|^4+2u^{\alpha-2}\Delta u|\nabla u|^2,$$
  $$u^{-\beta}\operatorname{Re}(u^{\beta-1}\lambda |\nabla u|^2u_1),^{\ 1}=\lambda\operatorname{Re}D_1u^1+(\beta+\alpha-1)\lambda\frac{|\nabla u|^4}{u^2}+2\lambda\Delta u\frac{|\nabla u|^2}{u},$$
  \begin{align}
    \begin{split}\label{A2}
      &u^{-\beta}\operatorname{Re}(u^{\beta-2}|\nabla u|^2\sqrt{-1}u_0u_1),^{\ 1}\\
      =&-\operatorname{Re}\frac{\sqrt{-1}u_0}{u}D_1u^1-\frac{|\nabla u|^2}{u^2}\operatorname{Re}G_1u^1-\frac\alpha 3(1-\frac 2 3\alpha)\frac{|\nabla u|^6}{u^4}\\
      &-\frac\alpha 3\Delta u\frac{|\nabla u|^4}{u^3}+(\alpha-1)\lambda\frac{|\nabla u|^4}{u^2}-2\frac{|\nabla u|^2}{u^2}u_0^2,    
    \end{split}
  \end{align}
  \begin{align*}
    &u^{-\beta}\operatorname{Re}(u^{\beta+\alpha-1}\sqrt{-1}u_0u_1),^{\ 1}\\
    =&-u^{\alpha-1}\operatorname{Re}G_1u^1-\frac\alpha 3(1-\frac 2 3\alpha)u^{\alpha-3}|\nabla u|^4\\
    &-\frac\alpha 3u^{\alpha-2}\Delta u|\nabla u|^2+(\alpha-1)\lambda u^{\alpha-1}|\nabla u|^2-u^{\alpha-1}u_0^2,
  \end{align*}
  \begin{align*}
    &u^{-\beta}\operatorname{Re}(u^\beta\lambda\sqrt{-1}u_0u_1),^{\ 1}\\
    =&-\lambda\operatorname{Re}G_1u^1-\frac\alpha 3(1-\frac 2 3\alpha)\lambda\frac{|\nabla u|^4}{u^2}-\frac\alpha 3\lambda\Delta u\frac{|\nabla u|^2}{u}+(\alpha-1)\lambda^2|\nabla u|^2-\lambda u_0^2,
  \end{align*}
  \begin{align*}
    &u^{-\beta}\operatorname{Re}(u^{\beta-1}u_0^2u_1),^{\ 1}\\
    =&-2\operatorname{Re}\frac{\sqrt{-1}u_0}{u}G_1u^1+(\beta+\frac 4 3\alpha-1)\frac{|\nabla u|^2}{u^2}u_0^2+\frac{\Delta u}{u}u_0^2.
  \end{align*}
  Use equation \eqref{equ}, and linearly combine seven identities together with coefficients
  $$\left\{\frac{\alpha}{3}(\frac 1 2-\frac \alpha 3),-\frac\alpha 6,\frac1 2-\frac \alpha 3,\frac 1 2(\beta+\frac 4 3\alpha-1),-1,1,1\right\},$$
  then identity \eqref{3} can be proved.
\end{proof}

The following identity is also useful, which provides a positive $u_0^2$ term.
\begin{proposition}
  If $n=1$, let $\beta$ be an undetermined constant, then
  \begin{align}
    \begin{split}\label{4}
      &u^{-\beta}\operatorname{Re}\left\{u^\beta\left[\frac 1 3(\frac 2 3\alpha-1)\frac{|\nabla u|^4}{u^3}-\frac{\sqrt{-1}u_0}{u^2}|\nabla u|^2\right]u_1\right\}_,^{\ 1}\\
      =&\frac 2 3(\frac 2 3\alpha-1)\frac{|\nabla u|^2}{u^2}\operatorname{Re}D_1u^1+\operatorname{Re}\frac{\sqrt{-1}u_0}{u}D_1u^1+\frac{|\nabla u|^2}{u^2}\operatorname{Re}G_1u^1\\
      &+(\frac{\beta+\alpha}{3}-1)(\frac 2 3\alpha-1)\frac{|\nabla u|^6}{u^4}-(\alpha-1)u^{\alpha-3}|\nabla u|^4+2\frac{|\nabla u|^2}{u^2}u_0^2.
    \end{split}
  \end{align}
\end{proposition}

\begin{proof}
  It's just $\displaystyle\frac 1 3(\frac 2 3\alpha-1)\times\eqref{A1}-\eqref{A2}$.
\end{proof}

In conclusion, inspired by the Riemannian case, the $\{(0,0,),2,6,+\}$ type identity \eqref{2} is found by the method of dimensional conservation and invariant tensors. As a supplement, \eqref{3} and \eqref{4} are proposed to deal with special case $n=1$.
\section{Proof of Theorem \ref{CR}}

In this section, we'll prove Theorem \ref{CR} by discussing four different cases. $\{(0,0),2,4,+\}$ type identity should be considered first, which is Case 1. The identity $\{(0,0),2,6,+\}$ type \eqref{2} deals with the "near-critical" exponent in Case 2. For $n=1$, Case 3 and Case 4 are finished with the help of \eqref{3} and \eqref{4} separately.

\vspace*{1em}
\noindent\textbf{Case 1.} $1<\alpha<\displaystyle\frac{n+2}{n+\displaystyle\frac{1}{2n}}$ for $n\geqslant2$.

By \eqref{d}, \eqref{e}, and use Lemma \ref{prep1} to write as square terms:
\begin{align}
  \begin{split}\label{case1}
    &\operatorname{Re}[(n-1)D_{ij}u^j+(n+2)E_{i\overline j}u^{\overline j}]_,^{\ i}\\
    =&(n+2)\sum_{i,j}|E_{i\overline j}|^2+(n-1)\sum_{i,j}|D_{ij}|^2\\
    &+2\alpha E_iu^i+2(n-1)\alpha(1-\frac{n\alpha}{n+2})\frac{|\nabla u|^4}{u^2}+(n-1)\mathscr R\\
    =&(n+2)\sum_{i,j}\left|E_{i\overline j}+\frac{\alpha}{n+2}L_{i\overline j}\right|^2\\
    &+(n-1)\left[\sum_{i,j}|D_{ij}|^2+\alpha\left(2-\frac{2n^2+1}{n(n+2)}\alpha\right)\frac{|\nabla u|^4}{u^2}+\mathscr R\right].
  \end{split}
\end{align}

If $1<\alpha<\displaystyle\frac{n+2}{n+\displaystyle\frac{1}{2n}}$, the RHS of identity \eqref{case1} is non-negative with positive $\displaystyle\frac{|\nabla u|^4}{u^2}$ term, and the LHS of identity \eqref{case1} is composed of divergence of vector fields, which are vanished while integrating over $M$. Hence $|\nabla u|^4=0$, i,e, $u$ is constant by integrating over $M$ on both sides of \eqref{case1}.

\begin{remark}
  The method discussed in Case 1 is similar to the one used by Xu in \cite{MR2568158}. Though the $\{(0,0),2,4,+\}$ type identity \eqref{case1} is simple, it's trivial when $n=1$ because of $E_{1\overline 1}=0$. Hence the case $n=1$ can't be taken into consideration in Case 1.
\end{remark}

\vspace*{1em}
\noindent\textbf{Case 2.} $\displaystyle\frac{n+2}{n+\displaystyle\frac{1}{2n}}\leqslant\alpha<\frac{n+2}{n}$ for $n\in\mathbb N^*$.

Consider the subcritical exponent case $\displaystyle\frac{n+2}{n+\displaystyle\frac{1}{2n}}\leqslant\alpha<\frac{n+2}{n}$ first. In identity \eqref{2}, take
$$d_1=e_1=\frac{n^2\alpha[3n+6-(n-1)\alpha]}{(2n+1)(n+2)^2},~~d_2=e_2=\frac{n\alpha}{n+2},~~d_3=e_3=n(\frac{n+1}{n+2}\alpha-1),$$
$$d_4=\frac{n}{2n+1}(3-\frac{7n+2}{n+2}\alpha),~~e_4=\frac{n(3+\alpha)}{2n+1},~~\mu=3,~~\beta=1-\alpha.$$
Rewrite all coefficients with the parameters above:
$$\Delta_1=\frac{2n^2\alpha[(4n+5)\alpha-3n-6]}{(2n+1)(n+2)^2}(1-\frac{n\alpha}{n+2}),~~\Theta_1=-\frac{6n^2\alpha(\alpha+n+2)}{(2n+1)(n+2)^2}(1-\frac{n\alpha}{n+2}),$$
$$\Xi_1=\frac{6n\alpha}{2n+1}(1-\frac{n\alpha}{n+2}),~~\Delta_3=\Theta_3=\frac{2n(\alpha-1)(2+n-n\alpha)}{2n+1},$$
$$\Delta_2=\Theta_2=\Xi_2=\Xi_3=\Delta_4=\Xi_4=0.$$

Notice that $\displaystyle\sum_{i,j,k}|D_{ij}u_{\overline k}+E_{i\overline k}u_j|^2=|\nabla u|^2\sum_{i,j}(|D_{ij}|^2+|E_{i\overline j}|^2)+2u^2\operatorname{Re}D_iE^i$. Use Lemma \ref{prep1} to rewrite identity \eqref{2} as square terms:

\begin{align}
  \begin{split}\label{case2}
   &u^{-\beta}\operatorname{Re}\Big\{u^\beta\Big[\left(d_1\frac{|\nabla u|^2}{u}+d_2 u^\alpha+d_3\lambda u\right)(D_i+E_i)\\
   &+n\sqrt{-1}u_0(d_4D_i+e_4E_i-3 G_i)\Big]\Big\}_,^{\ i}\\
   =&d_1u^{-2}\sum_{i,j,k}|D_{ij}u_{\overline k}+E_{i\overline k}u_j|^2+d_2u^{\alpha-1}\left[\sum_{i,j}(|D_{ij}|^2+|E_{i\overline j}|^2)+2\alpha(1-\frac{n\alpha}{n+2})\frac{|\nabla u|^4}{u^2}\right]\\
   &+d_3\lambda\Bigg[\sum_{i,j}\left(\left|D_{ij}+\frac{\Delta_3}{2d_3}\frac{u_iu_j}{u}\right|^2+\left|E_{i\overline j}+\frac{\Delta_3}{2d_3}L_{i\overline j}\right|^2\right)\\
   &+\left(2\alpha(1-\frac{n\alpha}{n+2})-\frac{2n-1}{n}\frac{\Delta_3^2}{4d_3^2}\right)\frac{|\nabla u|^4}{u^2}\Bigg]+\left[d_1\frac{|\nabla u|^2}{u^2}+d_2u^{\alpha-1}+d_3\lambda\right]\mathscr R+\mathbf Q_1,
  \end{split}
\end{align}
where $\mathbf Q_1$ can be written as a quadratic form:
\begin{align*}
  \mathbf Q_1=&d_1\sum_i|D_i|^2+d_1\sum_i|E_i|^2+3\sum_i|G_i|^2-d_4\operatorname{Re}D_iG^i-e_4\operatorname{Re}E_iG^i\\
  &+\Delta_1\frac{|\nabla u|^2}{u^2}\operatorname{Re}D_iu^i+\Theta_1\frac{|\nabla u|^2}{u^2}E_iu^i+\Xi_1\frac{|\nabla u|^2}{u^2}\operatorname{Re}G_iu^i+2d_1\alpha(1-\frac{n\alpha}{n+2})\frac{|\nabla u|^6}{u^4}.
\end{align*}

If $\displaystyle\alpha\in(\frac{n+2}{n+1},\frac{n+2}{n})$, $d_1=\displaystyle\frac{n\alpha}{n+2}\cdot\frac{n[3n+6-(n-1)\alpha]}{(2n+1)(n+2)}\geqslant\frac{n\alpha}{n+2}>0$, $d_2=\displaystyle\frac{n\alpha}{n+2}>0$, $d_3=\displaystyle n(\frac{n+1}{n+2}\alpha-1)>0$, which are correct for $\alpha\in[\displaystyle\frac{n+2}{n+\displaystyle\frac{1}{2n}},\frac{n+2}{n})$ of course.

Check the positivity of $\displaystyle\lambda\frac{|\nabla u|^4}{u^2}$ term:
$$d_3\left(2\alpha(1-\frac{n\alpha}{n+2})-\frac{2n-1}{n}\frac{\Delta_3^2}{4d_3^2}\right)=\frac{n}{{(2n+1)^2}d_3}(1-\frac{n\alpha}{n+2})f_1(\alpha),$$
where $f_1(\alpha)$ is a polynomial of $\alpha$:
\begin{align*}
  f_1(\alpha)=&\frac{n(10n^4+35n^3+44n^2+16n-6)}{(n+2)^2}\alpha^3-\frac{22n^4+57n^3+46n^2-8}{n+2}\alpha^2\\
  &+(14n^3+25n^2+8n-8)\alpha-(n+2)^2(2n-1).
\end{align*}
Study the monotonicity of $f_1$:
\begin{align*}
  f_1'(\alpha)=&3n(10n^4+35n^3+44n^2+16n-6)\left(\frac{\alpha}{n+2}\right)^2\\
  &-2(22n^4+57n^3+46n^2-8)\left(\frac{\alpha}{n+2}\right)+(14n^3+25n^2+8n-8).
\end{align*}
Compare the symmetry axis of $f_1'((n+2)x)$ with $\displaystyle\left(n+\frac {1}{2n}\right)^{-1}$, the minimum of $\displaystyle\frac{\alpha}{n+2}$:
\begin{align*}
  &\frac{22n^4+57n^3+46n^2-8}{3n(10n^4+35n^3+44n^2+16n-6)}\Big/\left(n+\frac{1}{2n}\right)^{-1}-1\\
  =&-\frac{16n^6+96n^5+150n^4+39n^3-66n^2+8}{6n^2(10n^4+35n^3+44n^2+16n-6)}<0,
\end{align*}
then $f_1'(\alpha)\geqslant\displaystyle f_1'(\frac{n+2}{n+\displaystyle\frac{1}{2n}})=\frac{(2n-1)(32n^5+96n^4+64n^3-41n^2-24n+8)}{(2n^2+1)^2}>0$, hence

$$f_1(\alpha)\geqslant f_1(\frac{n+2}{n+\displaystyle\frac{1}{2n}})=\frac{(n+2)(2n-1)(32n^5+16n^4-24n^3-28n^2+15n-2)}{(2n^2+1)^3}>0,$$
from which we have $2\displaystyle\alpha(1-\frac{n\alpha}{n+2})>\frac{2n-1}{n}\frac{\Delta_3^2}{4d_3^2}\geqslant0$, $\forall \alpha\in[\displaystyle\frac{n+2}{n+\displaystyle\frac{1}{2n}},\frac{n+2}{n})$.

Check the positivity of the quadratic form $\mathbf Q_1$, which corresponds to a matrix as
\begin{equation}\label{mat1}
  \begin{pmatrix}
   d_1 & 0 & \displaystyle-\frac{d_4}{2} & \displaystyle\frac{\Delta_1}{2}\\[10pt]
    0 & d_1 & \displaystyle-\frac{e_4}{2} & \displaystyle\frac{\Theta_1}{2}\\[10pt]
    \displaystyle-\frac{d_4}{2} & \displaystyle-\frac{e_4}{2} & 3 & \displaystyle\frac{\Xi_1}{2}\\[10pt]
   \displaystyle\frac{\Delta_1}{2} & \displaystyle\frac{\Theta_1}{2} & \displaystyle\frac{\Xi_1}{2} & \displaystyle2d_1\alpha(1-\frac{n\alpha}{n+2})
  \end{pmatrix}.
\end{equation}
Compute principal minor sequence of matrix \eqref{mat1}:
$$\begin{vmatrix}
  d_1 & 0 & \displaystyle-\frac{d_4}{2}\\[10pt]
  0 & d_1 & \displaystyle-\frac{e_4}{2}\\[10pt]
  \displaystyle-\frac{d_4}{2} & \displaystyle-\frac{e_4}{2} & 3
\end{vmatrix}=\frac{n^4\alpha(3-\displaystyle\frac{n-1}{n+2}\alpha)}{2(n+2)(2n+1)^3}f_2(\alpha),$$
where $3-\displaystyle\frac{n-1}{n+2}\alpha>3-\frac{n-1}{n+2}\cdot\frac{n+2}{n}=\frac{2n+1}{n}>0$,
\begin{align*}
  f_2(\alpha)=&-(37n^2+10n-2)\left(\frac{\alpha}{n+2}\right)^2+18(3n+1)\left(\frac{\alpha}{n+2}\right)-9\\
  \geqslant&\min\left\{f_2(\frac{n+2}{n+\displaystyle\frac{1}{2n}}),f_2(\frac{n+2}{n})\right\}\\
  =&\min\left\{\frac{32n^4+32n^3+80n^2+36n-9}{(2n^2+1)^2},\frac{2(2n+1)^2}{n^2}\right\}>0.
\end{align*}
The only left term we need to check is the determinant of $\mathbf Q_1$:
$$\det\eqref{mat1}=\frac{n^6\alpha^3(3-\displaystyle\frac{n-1}{n+2}\alpha)^2}{(n+2)^3(2n+1)^4}(1-\frac{n\alpha}{n+2})f_3(\alpha),$$
hence we only need to prove that $f_3(\alpha)$ is positive:
\begin{align*}
  f_3(\alpha)=&-2(2n-1)(11n^2+14n+2)\left(\frac{\alpha}{n+2}\right)^2+(79n^2+58n-2)\left(\frac{\alpha}{n+2}\right)-27n\\
  \geqslant&\min\left\{f_3(\frac{n+2}{n+\displaystyle\frac{1}{2n}}),f_3(\frac{n+2}{n})\right\}\\
  =&\min\left\{\frac{n(32n^4+96n^3+122n^2+132n-31)}{(2n^2+1)^2},\frac{2(n+2)(2n+1)^2}{n^2}\right\}>0.
\end{align*}

In conclusion, the matrix \eqref{mat1} is strictly positive definite. Then the RHS of identity \eqref{case2} are non-negative with some positive $\displaystyle\frac{|\nabla u|^6}{u^4}$ terms left while $\displaystyle\alpha\in[\frac{n+2}{n+\displaystyle\frac{1}{2n}},\frac{n+2}{n})$, hence $|\nabla u|^6=0$ by multiplying $u^\beta$ on both sides of \eqref{case2} and integrating over $M$, then $u$ is constant.

\begin{remark}
  For the critical exponent case $\alpha=\displaystyle\frac{n+2}{n}$, the parameters discussed above are
  $$d_k=e_k=1,~k=1,2,3,~~d_4=-2,~~e_4=2,~~\beta=-\frac 2 n,$$
  $$\Delta_l=\Theta_l=\Xi_l=0,~~l=1,2,3,4.$$
  Then the identity \eqref{case2} becomes Jerison-Lee type identity:
  \begin{align}
    \begin{split}\label{JL}
      &u^{\frac 2 n}\operatorname{Re}\left\{u^{-\frac 2 n}\left[\left(\frac{|\nabla u|^2}{u}+u^{\frac{n+2}{n}}+\lambda u\right)(D_i+E_i)-n\sqrt{-1}u_0(2D_i-2E_i+3G_i)\right]\right\}_,^{\ i}\\
      =&u^{-2}\sum_{i,j,k}|D_{ij}u_{\overline k}+E_{i\overline k}u_j|^2+\frac{|\nabla u|^2}{u^2}\mathscr R+(u^{\frac 2 n}+\lambda)\sum_{i,j}(|D_{ij}|^2+|E_{i\overline j}|^2+\mathscr R)\\
      &+\sum_i(|G_i+D_i|^2+|G_i-E_i|^2+|G_i|^2).
    \end{split}
  \end{align}
  The rigidity and the existence of non-trivial solutions can be deduced by discussing the positivity of $\mathscr R$. We recommend readers to see \cite{MR0924699} and \cite{MR3342188} for more details.
\end{remark}

\vspace*{1em}
\noindent\textbf{Case 3.} $1.06\leqslant\alpha<3$ for $n=1$.

It's a pity that the method in Case 1 failed when $n=1$. Besides, Case 2 can cover $2\leqslant\alpha<3$ with $n=1$ only. In this section, the identity \eqref{3} will be used to cover those difficulties. 

Notice that $E_{1\overline 1}=0$, $|\nabla u|^2|D_{11}|^2=u^2|D_1|^2$. Rewrite \eqref{2} while $n=1$:

\begin{align}
  \begin{split}\label{1dim}
    &u^{-\beta}\operatorname{Re}\left\{u^\beta\left[\left(d_1\frac{|\nabla u|^2}{u}+d_2 u^\alpha+d_3\lambda u\right)D_1+\sqrt{-1}u_0(d_4D_1-\mu G_1)\right]\right\}_,^{\ 1}\\
    =&\left(2d_1\frac{|\nabla u|^2}{u^2}+d_2u^{\alpha-1}+d_3\lambda\right)|D_{11}|^2+\left(d_1\frac{|\nabla u|^2}{u^2}+d_2u^{\alpha-1}+d_3\lambda\right)\\
    &\times\left[2\alpha(1-\frac\alpha 3)\frac{|\nabla u|^4}{u^2}+\mathscr R\right]+\mu|G_1|^2-d_4\operatorname{Re}D_1G^1\\
    &+\operatorname{Re}\left[\Delta_1\frac{|\nabla u|^2}{u^2}+\Delta_2u^{\alpha-1}+\Delta_3\lambda+\Delta_4\frac{\sqrt{-1}u_0}{u}\right]D_1u^1\\
    &+\operatorname{Re}\left[\Xi_1\frac{|\nabla u|^2}{u^2}+\Xi_2u^{\alpha-1}+\Xi_3\lambda+\Xi_4\frac{\sqrt{-1}u_0}{u}\right]G_1u^1.
  \end{split}
\end{align}
The coefficients are:
$$\Delta_1=(\beta+2\alpha-2)d_1+\frac \alpha 3\left(\frac 2 3\alpha-1\right)d_4,~~\Delta_2=-d_1+\left(\beta+2\alpha-1\right)d_2+\frac \alpha 3d_4,$$
$$\Delta_3=d_1+\left(\beta+\alpha\right)d_3+\left(\frac 2 3\alpha-1\right)d_4,~~\Delta_4=d_1+\left(\beta+\frac 5 3\alpha-1\right)d_4+\frac \alpha 3\left(\frac 2 3\alpha-1\right)\mu,$$
$$\Xi_1=3d_1-\frac \alpha 3\left(\frac 2 3\alpha-1\right)\mu,~~\Xi_2=3d_2-\frac \alpha 3\mu,~~\Xi_3=3d_3-\left(\frac 2 3\alpha-1\right)\mu,~~\Xi_4=3d_4-\beta\mu.$$

Take $d_1=\displaystyle\frac{\alpha}{36}(5\alpha-3)$, $d_2=d_3=\displaystyle\frac{\alpha-1}{2}$, $d_4=2-\displaystyle\frac 4 3\alpha$, $\mu=3$, $\beta=1-\alpha$. Rewrite all coefficients with the parameters above:
$$\Delta_1=\frac{\alpha}{108}(3-\alpha)(17\alpha-21),~~\Delta_2=\frac{\alpha}{12}(3-\alpha),~~\Delta_3=\frac{1}{12}(3-\alpha)(9\alpha-10),$$
$$\Delta_4=\frac{\alpha}{12}(3-\alpha),~~\Xi_1=\frac{\alpha}{4}(3-\alpha),~~\Xi_2=\frac{\alpha-3}{2},~~\Xi_3=\frac{3-\alpha}{2},~~\Xi_4=3-\alpha.$$

Consider $\eqref{1dim}+\displaystyle\frac{3-\alpha}{2}\times\eqref{3}$:
\begin{align}
  \begin{split}\label{case3}
    &u^{\alpha-1}\operatorname{Re}\Big\{u^{1-\alpha}\Big[\left(\frac{\alpha}{36}(5\alpha-3)\frac{|\nabla u|^2}{u}+\frac{\alpha-1}{2}(u^\alpha+\lambda u)\right)D_1\\
    &+\sqrt{-1}u_0\left((2-\frac 4 3\alpha)D_1-3 G_1\right)\\
    &+\frac{3-\alpha}{2}\left(\frac{\alpha}{3}(\frac 1 2-\frac \alpha 3)\frac{|\nabla u|^2}{u^2}-\frac\alpha 6u^{\alpha-1}+(\frac1 2-\frac \alpha 3)\lambda\right)\frac{|\nabla u|^2}{u}u_1\\
    &+\frac{3-\alpha}{2}\left(\frac \alpha 6\frac{|\nabla u|^2}{u^2}-u^{\alpha-1}+\lambda-\frac{\sqrt{-1}u_0}{u}\right)\sqrt{-1}u_0u_1\Big]\Big\}_,^{\ 1}\\
    =&\left(\frac{\alpha}{18}(5\alpha-3)\frac{|\nabla u|^2}{u^2}+\frac{\alpha-1}{2}(u^{\alpha-1}+\lambda)\right)|D_{11}|^2+3|G_1|^2+(\frac 4 3\alpha-2)\operatorname{Re}D_1G^1\\
    &+\frac{\alpha}{6}(\alpha-1)(\alpha+3)(1-\frac\alpha 3)\frac{|\nabla u|^6}{u^4}+\frac{\alpha}{36}(5\alpha-3)\frac{|\nabla u|^2}{u^2}\mathscr R\\
    &+\frac{\alpha-1}{2}(u^{\alpha-1}+\lambda)\left[2\alpha(1-\frac\alpha 3)\frac{|\nabla u|^4}{u^2}+\mathscr R\right]\\
    &+(3-\alpha)\operatorname{Re}\left[\frac{\alpha}{108}(5\alpha-3)\frac{|\nabla u|^2}{u^2}D_1u^1+\frac{7}{12}(\alpha-1)\lambda D_1u^1+\frac{\alpha}{6}\frac{|\nabla u|^2}{u^2}G_1u^1\right]\\
    =&\frac{\alpha-1}{2}\lambda\left[\left|D_{11}+\frac{7}{12}(3-\alpha)\frac{u_1u_1}{u}\right|^2+\frac{1}{144}(3-\alpha)(145\alpha-147)\frac{|\nabla u|^4}{u^2}+\mathscr R\right]\\
    &+\frac{\alpha-1}{2}u^{\alpha-1}\left[|D_{11}|^2+2\alpha(1-\frac\alpha 3)\frac{|\nabla u|^4}{u^2}+\mathscr R\right]+\frac{\alpha}{36}(5\alpha-3)\frac{|\nabla u|^2}{u^2}\mathscr R+\mathbf Q_2,
  \end{split}
\end{align}
where $\mathbf Q_2$ can be written as a quadratic form:
\begin{align*}
  \mathbf Q_2=&\frac{\alpha}{18}(5\alpha-3)|D_1|^2+3|G_1|^2+(\frac 4 3\alpha-2)\operatorname{Re}D_1G^1+(3-\alpha)\frac{|\nabla u|^2}{u^2}\\
  &\times\left[\frac{\alpha}{108}(5\alpha-3)\operatorname{Re}D_1u^1+\frac\alpha 6\operatorname{Re}G_1u^1+\frac{\alpha}{18}(\alpha-1)(\alpha+3)\frac{|\nabla u|^4}{u^2}\right].
\end{align*}

While $\alpha\in[1.06,3)$, $\alpha-1>0$, $(3-\alpha)(145\alpha-147)>0$, Hence we only need to check the positivity of quadratic form $\mathbf Q_2$, which corresponds to a matrix as
\begin{equation}\label{mat2}
  \begin{pmatrix}
    \displaystyle\frac{\alpha}{18}(5\alpha-3) & \displaystyle\frac 2 3\alpha-1 & \displaystyle\frac{\alpha}{216}(3-\alpha)(5\alpha-3)\\[10pt]
    \displaystyle\frac 2 3\alpha-1 & 3 & \displaystyle\frac{\alpha}{12}(3-\alpha)\\[10pt]
    \displaystyle\frac{\alpha}{216}(3-\alpha)(5\alpha-3) & \displaystyle\frac{\alpha}{12}(3-\alpha) & \displaystyle\frac{\alpha}{18}(3-\alpha)(\alpha-1)(\alpha+3)
  \end{pmatrix}.
\end{equation}
Compute principal minor sequence of matrix \eqref{mat2}:
$$\frac{\alpha}{18}(5\alpha-3)>\frac{\alpha}{9}>0,$$
$$\begin{vmatrix}
  \displaystyle\frac{\alpha}{18}(5\alpha-3) & \displaystyle\frac 2 3\alpha-1\\[10pt]
  \displaystyle\frac 2 3\alpha-1 & 3
\end{vmatrix}=\frac{1}{18}(\alpha+3)(7\alpha-6)>0,$$
$$\det\eqref{mat2}=\frac{\alpha}{5184}(3-\alpha)(3+\alpha)f_4(\alpha),$$
where $f_4(\alpha)=117\alpha^3+110\alpha^2-519\alpha+288$. Study the monotonicity of $f_4$:
$$f_4'(\alpha)=351\alpha^2+220\alpha-519>351+220-519=52>0,$$
then $f_4(\alpha)\geqslant f_4(1.06)=0.804872>0$, hence the matrix \eqref{mat2} is strictly positive definite.

In conclusion, the RHS of identity \eqref{case3} are non-negative with some positive $\displaystyle u^{\alpha-3}|\nabla u|^4$ terms left while $\alpha\in[1.06,2)$. Then $|\nabla u|^4=0$, i.e. $u$ is a constant, by multiplying $u^\beta$ on both sides of \eqref{case3} and integrating over $M$.

\begin{remark}
  In the critical exponent case $\alpha=3$, the identity \eqref{case3} becomes \eqref{JL} in the case $n=1$ again. Besides, the lower bound of 1.06 can be decreased to 1.052327, the approximation of
  $$\frac{2}{351}\left\{\sqrt{194269}\cos\left[\frac 1 3\left(\arccos\frac{84611717}{194269^{\frac 3 2}}-\pi\right)\right]-55\right\},$$
  which is the biggest root of $f_4$. However, we'll use a new identity in Case 4, hence it's meaningless to compute such finely.
\end{remark}

\vspace*{1em}
\noindent\textbf{Case 4.} $1<\alpha\leqslant1.06$ for $n=1$.

Take $d_1=\displaystyle\frac{1}{18}$, $d_2=d_3=\displaystyle\frac{\alpha-1}{2}$, $d_4=\displaystyle\frac 2 3$, $\mu=3$, $\beta=\displaystyle\frac 1 2$ in identity \eqref{1dim}. Rewrite all coefficients in \eqref{1dim} with parameters above:
$$\Delta_1=\displaystyle\frac{1}{108}(16\alpha^2-12\alpha-9),~~\Delta_2=\frac{1}{36}(4\alpha-1)(9\alpha-7),~~\Delta_3=\frac{1}{36}(18\alpha^2+7\alpha-31),$$
$$\Delta_4=\frac{1}{18}(12\alpha^2+2\alpha-5),~~\Xi_1=-\frac{1}{6}(4\alpha^2-6\alpha-1),~~\Xi_2=\frac{\alpha-3}{2},~~\Xi_3=\frac{3-\alpha}{2},~~\Xi_4=\frac 1 2.$$

Consider $\eqref{1dim}+\displaystyle\frac{3-\alpha}{2}\times\eqref{3}+\frac{9}{40}\times\eqref{4}$ with $\beta=\displaystyle\frac 1 2$:
\begin{align}
  \begin{split}\label{case4}
    &u^{-\frac 1 2}\operatorname{Re}\Big\{u^{\frac 1 2}\Big[\left(\frac{1}{18}\frac{|\nabla u|^2}{u}+\frac{\alpha-1}{2}(u^\alpha+\lambda u)\right)D_1+\sqrt{-1}u_0\left(\frac 2 3D_1-3 G_1\right)\\
    &+\frac{3-\alpha}{2}\left[\left(-\frac\alpha 6u^{\alpha-1}+(\frac1 2-\frac \alpha 3)\lambda\right)\frac{|\nabla u|^2}{u}+\left(-u^{\alpha-1}+\lambda-\frac{\sqrt{-1}u_0}{u}\right)\sqrt{-1}u_0\right]u_1\\
    &+\left((2\alpha-3)(\frac{1}{36}\alpha^2-\frac{1}{12}\alpha+\frac{1}{40})\frac{|\nabla u|^2}{u}-(\frac 1 3\alpha^2-\frac 9 8\alpha+\frac 3 5)\sqrt{-1}u_0\right)\frac{|\nabla u|^2}{u^2}u_1\Big]\Big\}_,^{\ 1}\\
    =&\left[\frac 1 9\frac{|\nabla u|^2}{u^2}+\frac{\alpha-1}{2}(u^{\alpha-1}+\lambda)\right]|D_{11}|^2+3|G_1|^2-\frac 2 3\operatorname{Re}D_1G^1\\
    &+\frac{1}{2160}(80\alpha^4-600\alpha^3+1468\alpha^2-1272\alpha+405)\frac{|\nabla u|^6}{u^4}+\frac{1}{18}\frac{|\nabla u|^2}{u^2}\mathscr R\\
    &+\frac{\alpha-1}{2}u^{\alpha-1}\left[-\frac{1}{60}(40\alpha^2-120\alpha+27)\frac{|\nabla u|^4}{u^2}+\mathscr R\right]\\
    &+\frac{\alpha-1}{2}\lambda\left[2\alpha(1-\frac\alpha 3)\frac{|\nabla u|^4}{u^2}+\mathscr R\right]+\frac{9}{20}\frac{|\nabla u|^2}{u^2}u_0^2\\
    &+\frac{1}{270}(30\alpha^3-95\alpha^2+132\alpha-63)\frac{|\nabla u|^2}{u^2}\operatorname{Re}D_1u^1\\
    &+(\alpha-1)\left[\frac{1}{36}(39-7\alpha)u^{\alpha-1}+\frac 1 9(6\alpha+1)\lambda\right]\operatorname{Re}D_1u^1\\
    &+\frac{1}{360}(360\alpha^2-365\alpha+116)\operatorname{Re}\frac{\sqrt{-1}u_0}{u}D_1u^1\\
    &-\frac{1}{120}(40\alpha^2+15\alpha-92)\frac{|\nabla u|^2}{u^2}\operatorname{Re}G_1u^1+(\alpha-\frac 5 2)\operatorname{Re}\frac{\sqrt{-1}u_0}{u}G_1u^1\\
    =&\frac{\alpha-1}{2}u^{\alpha-1}\left[\left|D_{11}+\frac{1}{36}(39-7\alpha)\frac{u_1u_1}{u}\right|^2-\frac{4565\alpha^2-15690\alpha+10521}{6480}\frac{|\nabla u|^4}{u^2}\right]\\
    &+\frac{\alpha-1}{2}\lambda\left[\left|D_{11}+\frac1 9(6\alpha+1)\frac{u_1u_1}{u}\right|^2-\frac{90\alpha^2-150\alpha+1}{81}\frac{|\nabla u|^4}{u^2}\right]\\
    &+\left[\frac{1}{18}\frac{|\nabla u|^2}{u^2}+\frac{\alpha-1}{2}(u^{\alpha-1}+\lambda)\right]\mathscr R+\mathbf Q_3,
  \end{split}
\end{align}
where $\mathbf Q_3$ can be written as a quadratic form:
\begin{align*}
  \mathbf Q_3=&\frac 1 9|D_1|^2+3|G_1|^2-\frac 2 3\operatorname{Re}D_1G^1+\frac{9}{20}\frac{|\nabla u|^2}{u^2}u_0^2+\Delta_4'\operatorname{Re}\frac{\sqrt{-1}u_0}{u}D_1u^1\\
  &+\Xi_4'\operatorname{Re}\frac{\sqrt{-1}u_0}{u}G_1u^1+\Delta_1'\frac{|\nabla u|^2}{u^2}\operatorname{Re}D_1u^1+\Xi_1'\frac{|\nabla u|^2}{u^2}\operatorname{Re}G_1u^1+A\frac{|\nabla u|^6}{u^4}
\end{align*}
with the coefficients as
$$\Delta_1'=\frac{1}{270}(30\alpha^3-95\alpha^2+132\alpha-63),~~\Delta_4'=\frac{1}{360}(360\alpha^2-365\alpha+116),$$
$$\Xi_1'=-\frac{1}{120}(40\alpha^2+15\alpha-92),~~\Xi_4'=\alpha-\frac 5 2,$$
$$A=\frac{1}{2160}(80\alpha^4-600\alpha^3+1468\alpha^2-1272\alpha+405).$$

While $\alpha\in(1,1.06]$, check the positivity of $\displaystyle u^{\alpha-1}\frac{|\nabla u|^4}{u^2}$ term and $\displaystyle \lambda\frac{|\nabla u|^4}{u^2}$ term:
$$-(4565\alpha^2-15690\alpha+10521)\geqslant-(4565\times1-15690\times1+10521)=604>0,$$
$$-(90\alpha^2-150\alpha+1)\geqslant-(90\times1.06^2-150+1)=47.876>0,$$
Hence we only need to check the positivity of quadratic form $\mathbf Q_2$, which corresponds to a matrix as
\begin{equation}\label{mat3}
  \begin{pmatrix}
    \displaystyle\frac 1 9 & \displaystyle-\frac 1 3 & \displaystyle\frac{\Delta_4'}{2} & \displaystyle\frac{\Delta_1'}{2}\\[10pt]
    \displaystyle-\frac 1 3 & 3 & \displaystyle\frac{\Xi_4'}{2} & \displaystyle\frac{\Xi_1'}{2}\\[10pt]
    \displaystyle\frac{\Delta_4'}{2} & \displaystyle\frac{\Xi_4'}{2} & \displaystyle\frac{9}{20} & 0\\[10pt]
    \displaystyle\frac{\Delta_1'}{2} & \displaystyle\frac{\Xi_1'}{2} & 0 & A
  \end{pmatrix}.
\end{equation}
Compute principal minor sequence of matrix \eqref{mat3}:
$$\begin{vmatrix}
  \displaystyle\frac 1 9 & \displaystyle-\frac 1 3\\[10pt]
  \displaystyle-\frac 1 3 & 3
\end{vmatrix}=\frac 2 9>0,~~\begin{vmatrix}
  \displaystyle\frac 1 9 & \displaystyle-\frac 1 3 & \displaystyle\frac{\Delta_4'}{2}\\[10pt]
  \displaystyle-\frac 1 3 & 3 & \displaystyle\frac{\Xi_4'}{2}\\[10pt]
  \displaystyle\frac{\Delta_4'}{2} & \displaystyle\frac{\Xi_4'}{2} & \displaystyle\frac{9}{20}
\end{vmatrix}=\frac{f_5(\alpha)}{57600},~~\det\eqref{mat3}=\frac{f_6(\alpha)}{29859840000},$$
where $f_5(\alpha)=-43200\alpha^4+78000\alpha^3-40115\alpha^2+8800\alpha-992$,
\begin{align*}
  f_6(\alpha)=&-460800000\alpha^8+6320640000\alpha^7-25055552000\alpha^6+44595172000\alpha^5\\
  &-42848423575\alpha^4+24660626800\alpha^3-8756098960\alpha^2+1823449600\alpha-252801536.
\end{align*}

Study the concavity of $f_5$:
$$f_5''(\alpha)=10(-51840\alpha^2+46800\alpha-8023)<10(-51840+46800\times1.06-8023)=-102550<0,$$
then $f_5(\alpha)\geqslant\min\{f_5(1),f_5(1.06)\}=\min\{2493,1623.03\}>0$.

Similarly, check the positivity of $f_6$:
\begin{align*}
  f_6''(\alpha)=-20(&1290240000\alpha^6-13273344000\alpha^5+37583328000\alpha^4\\
  &-44595172000\alpha^3+25709054145\alpha^2-7398188040\alpha+875609896),
\end{align*}
\begin{align*}
  f_6^{(3)}(\alpha)=-600 (&258048000\alpha^5-2212224000\alpha^4+5011110400\alpha^3\\
  &-4459517200\alpha^2+1713936943\alpha-246606268),
\end{align*}
\begin{align*}
  f_6^{(4)}(\alpha)=-600(&1290240000\alpha^4-8848896000\alpha^3\\
  &+15033331200\alpha^2-8919034400\alpha+1713936943),
\end{align*}
$$f_6^{(5)}(\alpha)=-480000 (6451200\alpha^3-33183360\alpha^2+37583328\alpha-11148793),$$
$$f_6^{(6)}(\alpha)=46080000 (-201600\alpha^2+691320\alpha-391493)  \geqslant f_6^{(6)}(1)=46080000\times98227>0,$$
$$f_6^{(5)}(\alpha)\geqslant f_6^{(5)}(1)=480000 \times297625>0,$$
$$f_6^{(4)}(\alpha)\leqslant f_6^{(4)}(1.06)=-600\times240932969.8544<0,$$
$$f_6^{(3)}(\alpha)\leqslant f_6^{(3)}(1)=-600\times64747875<0,$$
$$f_6''(\alpha)\leqslant f_6''(1)=-20\times191528001<0,$$
then $f_6$ is concave while $\alpha\in(1,1.06]$, hence
$$f_6(\alpha)\geqslant\min\{f_6(1),f_6(1.06)\}=\min\{26212329,2.38\times10^7\}>0.$$

In conclusion, the matrix \eqref{mat3} is strictly positive definite, hence the RHS of identity \eqref{case4} are non-negative with some positive $\displaystyle u^{\alpha-3}|\nabla u|^4$ terms left while $\alpha\in(1,1.06]$. Then $|\nabla u|^4=0$, i.e. $u$ is a constant, by multiplying $u^\beta$ on both sides of \eqref{case4} and integrating over $M$.

\vspace*{1em}

Combine Case 1 $\sim$ Case 4, then Theorem \ref{CR} is proved. $\hfill\qed$
\section{Theorem \ref{JL2}: answer to the problem raised by Jerison-Lee}

In \cite{MR0924699}, Jerison-Lee found a three-dimensional family of differential identities with positive RHS for the Yamabe equation on Heisenberg group $\mathbb H^n$ by using the computer. However, they care about whether there exists a theoretical framework that would predict the existence and the structure of such formulae.

In this section, we unravel the mystery of the existence of these identities and prove that all useful identities of $\{(0,0),2,6,+\}$ type must be the three-dimensional family in \cite{MR0924699}. \textbf{From now on}, $(M^{2n+1},\theta)=\mathbb H^n$ is Heisenberg group, then $h_{i\overline j}=\delta_{i\overline j}$, $R_{i\overline j}=0$. All the couple indices, such as $i$ and $\overline i$, will be considered as summation indices taking part in the process of summing from 1 to $n$. We study the Yamabe equation:
\begin{equation}\label{Hequ}
  \Delta u+u^{\frac{n+2}{n}}=0~~\text{on}~~\mathbb H^n,
\end{equation}
which is identical with \eqref{equ} in the case $\lambda=0$ and $\alpha=\displaystyle\frac{n+2}{n}$. Then, Lemma \ref{invariance} becomes

$$D_{i,\overline i}=u^{-1}\sum_{i,j=1}^n|D_{ij}|^2+\frac 2 n\frac{D_iu_{\overline i}}{u}-\frac{n+2}{n}\frac{E_iu_{\overline i}}{u}+\frac{n+2}{n}\frac{G_iu_{\overline i}}{u},$$
$$E_{i,\overline i}=u^{-1}\sum_{i,j=1}^n|E_{i\overline j}|^2-\frac{n-1}{n}\frac{D_{\overline i}u_i}{u}+\frac 1 n\frac{E_{\overline i}u_i}{u}-\frac{n-1}{n}\frac{G_{\overline i}u_i}{u},$$
$$\operatorname{Im}G_{i,\overline i}=\operatorname{Im}\left[\frac 1 n\frac{D_{\overline i}u_i}{u}+\frac{n+1}{n}\frac{G_{\overline i}u_i}{u}\right].$$

In the critical exponent case, target identities are composed of divergence of some vector fields and positive quadratic form of invariant tensors only. Similar with discussion about $\displaystyle\sum_i|G_i|^2$ in Section 2, we need $\{(0,0),2,6,+\}$ type identity. Rewrite \eqref{2} first:
\begin{align}
  \begin{split}\label{2'}
    &u^{-\beta}\operatorname{Re}\Big\{u^\beta\Big[\left(d_1\frac{|\nabla u|^2}{u}+d_2 u^{\frac{n+2}{n}}+d_4n\sqrt{-1}u_0\right)D_i\\
    &+\left(e_1\frac{|\nabla u|^2}{u}+e_2 u^{\frac{n+2}{n}}+e_4n\sqrt{-1}u_0\right)E_i-\mu n\sqrt{-1}u_0G_i\Big]\Big\}_{,\overline i}\\
    =&\left[d_1\frac{|\nabla u|^2}{u^2}+d_2u^{\frac 2 n}\right]\sum_{i,j}|D_{ij}|^2+d_1\sum_i|D_i|^2+\left[e_1\frac{|\nabla u|^2}{u^2}+e_2u^{\frac 2 n}\right]\sum_{i,j}|E_{i\overline j}|^2\\
    &+e_1\sum_i|E_i|^2+\mu\sum_i|G_i|^2+(d_1+e_1)\operatorname{Re}D_iE_{\overline i}-d_4\operatorname{Re}D_iG_{\overline i}-e_4\operatorname{Re}E_iG_{\overline i}\\
    &+\operatorname{Re}\left[\Delta_1\frac{|\nabla u|^2}{u^2}+\Delta_2u^{\frac 2 n}+\Delta_4\frac{n\sqrt{-1}u_0}{u}\right]D_iu_{\overline i}+\left[\Theta_1\frac{|\nabla u|^2}{u^2}+\Theta_2u^{\frac 2 n}\right]E_iu_{\overline i}\\
    &+\operatorname{Re}\left[\Xi_1\frac{|\nabla u|^2}{u^2}+\Xi_2u^{\frac 2 n}+\Xi_4\frac{n\sqrt{-1}u_0}{u}\right]G_iu_{\overline i}.
  \end{split}
\end{align}
The coefficients are:
$$\Delta_1=(\beta+\frac{n+3}{n})d_1-\frac{n-1}{n}e_1+\frac 1 nd_4,$$
$$\Delta_2=-\frac 1 nd_1+(\beta+\frac{n+4}{n})d_2-\frac{n-1}{n}e_2+\frac 1 nd_4,$$
$$\Delta_4=\frac 1 nd_1+(\beta+\frac{n+3}{n})d_4+\frac{n-1}{n}e_4+\frac 1 n\mu,$$
$$\Theta_1=-\frac{n+2}{n}d_1+(\beta+\frac{n+2}{n})e_1+\frac 1 ne_4,$$
$$\Theta_2=-\frac 1 ne_1-\frac{n+2}{n}d_2+(\beta+\frac{n+3}{n})e_2+\frac 1 n e_4,$$
$$\Xi_1=\frac{n+2}{n}d_1-\frac{n-1}{n}e_1-\frac 1 n\mu,$$
$$\Xi_2=\frac{n+2}{n}d_2-\frac{n-1}{n}e_2-\frac 1 n\mu,$$
$$\Xi_4=\frac{n+2}{n}d_4+\frac{n-1}{n}e_4-\beta\mu.$$

Besides, notice that
\begin{align*}
    &\operatorname{Re}[u_{jk,\overline i}u_{\overline j}u_{\overline k}u_i-u_{j\overline k,\overline i}u_{\overline j}u_ku_i]\\
    =&\operatorname{Re}[u_{j\overline i,k}u_{\overline j}u_{\overline k}u_i+2\sqrt{-1}u_{j0}u_{\overline j}|\nabla u|^2-u_{j\overline k,\overline i}u_{\overline j}u_ku_i]\\
    =&\operatorname{Re}[u_{\overline ij,k}u_{\overline j}u_{\overline k}u_i+4\sqrt{-1}u_{0i}u_{\overline i}|\nabla u|^2-u_{j\overline k,\overline i}u_{\overline j}u_ku_i]\\
    =&4|\nabla u|^2\operatorname{Re}\sqrt{-1}u_{0i}u_{\overline i}
\end{align*}
then
\begin{align}
  \begin{split}\label{ijk}
    &\operatorname{Re}[u_{jk}u_{\overline j}u_{\overline k}u_i-u_{j\overline k}u_{\overline j}u_ku_i]_{,\overline i}\\
    =&\operatorname{Re}[u_{jk,\overline i}u_{\overline j}u_{\overline k}u_i-u_{j\overline k,\overline i}u_{\overline j}u_ku_i]+\operatorname{Re}u_{jk}u_{\overline j\overline i}u_{\overline k}u_i+\operatorname{Re}u_{jk}u_{\overline j}u_{\overline k\overline i}u_i\\
    &-\operatorname{Re}u_{j\overline k}u_{\overline j\overline i}u_ku_i-\operatorname{Re}u_{j\overline k}u_{\overline j}u_{k\overline i}u_i+u_{jk}u_{\overline j}u_{\overline k}u_{i\overline i}-\operatorname{Re}u_{j\overline k}u_{\overline j}u_ku_{i\overline i}\\
    =&4|\nabla u|^2\operatorname{Re}\sqrt{-1}u_{0i}u_{\overline i}+2\sum_j|u_{jk}u_{\overline k}|^2-\operatorname{Re}u_{j\overline k}u_ku_{\overline j\overline i}u_i-\sum_k|u_{j\overline k}u_{\overline j}|^2\\
    &-2\operatorname{Re}\sqrt{-1}u_0u_{i\overline j}u_{\overline i}u_j+\Delta u\operatorname{Re}u_{ij}u_{\overline i}u_{\overline j}+\operatorname{Re}n\sqrt{-1}u_0u_{ij}u_{\overline i}u_{\overline j}\\
    &-\Delta u\operatorname{Re}u_{i\overline j}u_{\overline i}u_j-n\operatorname{Re}\sqrt{-1}u_0u_{i\overline j}u_{\overline i}u_j\\
    =&2u^2\sum_i|D_i|^2-u^2\sum_i|E_i|^2-u^2\operatorname{Re}D_iE_{\overline i}+\frac 4 n|\nabla u|^2\operatorname{Re}G_iu_{\overline i}\\
    &+\operatorname{Re}\left[\frac{3(n+3)}{n}|\nabla u|^2+\frac{n-1}{n}u\Delta u+(n+1)\sqrt{-1}uu_0\right]D_iu_{\overline i}\\
    &-\left[3|\nabla u|^2+\frac{n+2}{n}u\Delta u\right]E_iu_{\overline i}+\frac{9n+5}{n^2}\frac{|\nabla u|^6}{u^2}+\frac{4}{n^2}\frac{|\nabla u|^4\Delta u}{u}\\
    &-\frac{n+1}{n^2}|\nabla u|^2(\Delta u)^2+(n+1)|\nabla u|^2u_0^2.
  \end{split}
\end{align}
Hence another $\{(0,0),2,6,+\}$ type identity is found:

\begin{proposition}
  Let $\beta$ be an undetermined constant, then
  \begin{align}
    \begin{split}\label{1'}
      &u^{-\beta}\operatorname{Re}[u^{\beta-1}(D_ju_{\overline j}-E_ju_{\overline j})u_i]_{,\overline i}\\
      =&2\sum_i|D_i|^2-\sum_i|E_i|^2-\operatorname{Re}D_iE_{\overline i}+\frac 3 n\frac{|\nabla u|^2}{u^2}\operatorname{Re}G_iu_{\overline i}\\
      &+\operatorname{Re}\left[(\beta+\frac{n+3}{n})\frac{|\nabla u|^2}{u^2}-u^{\frac 2 n}+\frac{n\sqrt{-1}u_0}{u}\right]D_iu_{\overline i}\\
      &-\left[(\beta+\frac{n+6}{n})\frac{|\nabla u|^2}{u^2}-\frac{n+1}{n}u^{\frac 2 n}\right]\operatorname{Re}E_iu_{\overline i},
    \end{split}
  \end{align}
\end{proposition}

\begin{proof}
  By Lemma \ref{prep2} and \eqref{ijk}, we yield that
  \begin{align*}
    &\operatorname{Re}[D_{jk}u_{\overline j}u_{\overline k}u_i-E_{j\overline k}u_{\overline j}u_ku_i]_{,\overline i}\\
    =&\operatorname{Re}[u_{jk}u_{\overline j}u_{\overline k}u_i-u_{j\overline k}u_{\overline j}u_ku_i]_{,\overline i}+\operatorname{Re}\left[-\frac 3 n\frac{|\nabla u|^4}{u}u_i+\frac 1 n\Delta u|\nabla u|^2u_i+\sqrt{-1}u_0|\nabla u|^2u_i\right]_{,\overline i}\\
    =&2u^2\sum_i|D_i|^2-u^2\sum_i|E_i|^2-u^2\operatorname{Re}D_iE_{\overline i}+\frac 3 n|\nabla u|^2\operatorname{Re}G_iu_{\overline i}\\
    &+\operatorname{Re}\left[\frac{3(n+1)}{n}|\nabla u|^2+u\Delta u+n\sqrt{-1}uu_0\right]D_iu_{\overline i}-\left[\frac{3(n+2)}{n}|\nabla u|^2+\frac{n+1}{n}u\Delta u\right]E_iu_{\overline i},
  \end{align*}
  then \eqref{1'} is proved by inserting $u^{\beta-2}$ into vector field.
\end{proof}

\begin{remark}
  It's noteworthy that, the \eqref{1'} type identities in general CR manifolds are omitted, because some Webster curvature terms occur. Without other assumptions of Webster curvature, those terms are tricky.
\end{remark}

To seek for all $\{(0,0),2,6,+\}$ type identities with invariant tensors as RHS, the vector fields composed of non-invariant things are also needed. Let $\beta$ be an undetermined constant, and consider
\begin{align}
  \begin{split}\label{D2}
    &u^{-\beta}\operatorname{Re}[u^{\beta-3}|\nabla u|^4u_i]_{,\overline i}\\
    =&2\frac{|\nabla u|^2}{u^2}(\operatorname{Re}D_iu_{\overline i}+E_iu_{\overline i})+(\beta+\frac{n+2}{n})\frac{|\nabla u|^6}{u^4}-\frac{n+2}{n}u^{\frac 2n-2}|\nabla u|^4,
  \end{split}
\end{align}
\begin{align}
  \begin{split}\label{D1L1}
    &u^{-\beta}\operatorname{Re}[u^{\beta+\frac 2n-1}|\nabla u|^2u_i]_{,\overline i}\\
    =&u^{\frac 2 n}(\operatorname{Re}D_iu_{\overline i}+E_iu_{\overline i})+(\beta+\frac{n+3}{n})u^{\frac 2n-2}|\nabla u|^4-\frac{n+1}{n}u^{\frac 4 n}|\nabla u|^2,
  \end{split}
\end{align}
\begin{align}
  \begin{split}\label{D1T1}
    &u^{-\beta}\operatorname{Re}[u^{\beta-2}|\nabla u|^2\cdot n\sqrt{-1}u_0u_i]_{,\overline i}\\
    =&-\operatorname{Re}\frac{n\sqrt{-1}u_0}{u}D_iu_{\overline i}-\frac{|\nabla u|^2}{u^2}\operatorname{Re}G_iu_{\overline i}+\frac 1 n\frac{|\nabla u|^6}{u^4}\\
    &+\frac 1 nu^{\frac 2n-2}|\nabla u|^4-n(n+1)\frac{|\nabla u|^2u_0^2}{u^2},
  \end{split}
\end{align}
$$u^{-\beta}\operatorname{Re}[u^{\beta+\frac 4n+1}u_i]_{,\overline i}=(\beta+\frac{n+4}{n})u^{\frac 4 n}|\nabla u|^2-u^{\frac{2n+6}{n}},$$
\begin{align}
  \begin{split}\label{L1T1}
    &u^{-\beta}\operatorname{Re}[u^{\beta+\frac 2 n}\cdot n\sqrt{-1}u_0u_i]_{,\overline i}\\
    =&-u^{\frac 2 n}\operatorname{Re}G_iu_{\overline i}+\frac 1 nu^{\frac 2n-2}|\nabla u|^4+\frac 1 nu^{\frac 4 n}|\nabla u|^2-n^2u^{\frac 2 n}u_0^2,
  \end{split}
\end{align}
\begin{align}
  \begin{split}\label{T2}
    &u^{-\beta}\operatorname{Re}[u^{\beta-1}n^2 u_0^2u_i]_{,\overline i}\\
    =&-2\operatorname{Re}\frac{n\sqrt{-1}u_0}{u}G_iu_{\overline i}+n(n\beta+n+2)\frac{|\nabla u|^2u_0^2}{u^2}-n^2u^{\frac 2 n}u_0^2.
  \end{split}
\end{align}
Eliminate all terms except for invariant tensor terms, we yield the following identity:
\begin{proposition}
  Let $\beta$ be an undetermined constant, then
  \begin{align}
    \begin{split}\label{3'}
      &u^{-\beta}\operatorname{Re}\Big\{u^{\beta-1}\Big[\frac{|\nabla u|^4}{u^2}+u^{\frac 2 n}|\nabla u|^2-(n\beta+n+2)\frac{|\nabla u|^2\cdot n\sqrt{-1}u_0}{u}\\
      &\quad\quad+(n+1)u^{\frac{n+2}{n}}\cdot n\sqrt{-1}u_0-(n+1)n^2 u_0^2\Big]u_i\Big\}_{,\overline i}\\
      =&\operatorname{Re}\left[2\frac{|\nabla u|^2}{u^2}+u^{\frac 2 n}+(n\beta+n+2)\frac{n\sqrt{-1}u_0}{u}\right]D_iu_{\overline i}+\left[2\frac{|\nabla u|^2}{u^2}+u^{\frac 2 n}\right]E_iu_{\overline i}\\
      &+\operatorname{Re}\left[(n\beta+n+2)\frac{|\nabla u|^2}{u^2}-(n+1)u^{\frac 2 n}+2(n+1)\frac{n\sqrt{-1}u_0}{u}\right]G_iu_{\overline i}.
    \end{split}
  \end{align}
\end{proposition}

\begin{proof}
  $\eqref{D2}+\eqref{D1L1}-(n\beta+n+2)\times\eqref{D1T1}+(n+1)\times\eqref{L1T1}-(n+1)\times\eqref{T2}$.
\end{proof}

\begin{remark}
  Because of $u^{\frac{2n+6}{n}}$ term, vector field $u^{-\beta}\operatorname{Re}[u^{\beta+\frac 4n+1}u_i]_{,\overline i}$ is useless. Besides, identity \eqref{3'} with $n=1$ is identical to \eqref{3} with $\alpha=3$ and $\lambda=0$.
\end{remark}

Similar as Riemannian $\{(0,0),2,4,+\}$ case in \cite{MR1134481} and \cite{MR3229793}, all reasonable $\{(0,0),2,6,+\}$ type identities with invariant tensors as RHS are found. Linearly combine \eqref{2'}, \eqref{1'} and \eqref{3'} together. Since we hope that a non-trivial solution exists in the critical exponent case, the cross terms must vanish. Here the cross terms are:
$$\frac{|\nabla u|^2}{u^2}\operatorname{Re}D_iu_{\overline i},~~u^{\frac 2 n}\operatorname{Re}D_iu_{\overline i},~~\operatorname{Re}\frac{n\sqrt{-1}u_0}{u}D_iu_{\overline i},~~\frac{|\nabla u|^2}{u^2}E_iu_{\overline i},$$
$$u^{\frac 2 n}E_iu_{\overline i},~~\frac{|\nabla u|^2}{u^2}\operatorname{Re}G_iu_{\overline i},~~u^{\frac 2 n}\operatorname{Re}G_iu_{\overline i},~~\operatorname{Re}\frac{n\sqrt{-1}u_0}{u}G_iu_{\overline i}.$$
If not, take some $D_i$ term for example, then we'll yield that $D_{ij}+c\displaystyle\frac{u_iu_j}{u}=0$ for some $c\neq0$ by writing into a complete square form. However, $D_{ij}+c\displaystyle\frac{u_iu_j}{u}$ is not an invariant tensor, i.e. $0=\left(D_{ij}+c\displaystyle\frac{u_iu_j}{u}\right)_{,\overline i}$ is composed by some non-invariant things, which will cause that $u$ can only be a constant.

With the idea above, linearly combine RHS of \eqref{2'}, \eqref{1'} and \eqref{3'}:
\begin{align}
  \begin{split}\label{high}
    &\text{RHS of }[\eqref{2'}+a\times\eqref{1'}+b\times\eqref{3'}]\\
    =&\left[d_1\frac{|\nabla u|^2}{u^2}+d_2u^{\frac 2 n}\right]\sum_{i,j}|D_{ij}|^2+(d_1+2a)\sum_i|D_i|^2+\left[e_1\frac{|\nabla u|^2}{u^2}+e_2u^{\frac 2 n}\right]\sum_{i,j}|E_{i\overline j}|^2\\
    &+(e_1-a)\sum_i|E_i|^2+\mu\sum_i|G_i|^2+(d_1+e_1-a)\operatorname{Re}D_iE_{\overline i}-d_4\operatorname{Re}D_iG_{\overline i}\\
    &-e_4\operatorname{Re}E_iG_{\overline i}+\operatorname{Re}\left[\widetilde{\Delta}_1\frac{|\nabla u|^2}{u^2}+\widetilde{\Delta}_2u^{\frac 2 n}+\widetilde{\Delta}_4\frac{n\sqrt{-1}u_0}{u}\right]\operatorname{Re}D_iu_{\overline i}\\
    &+\left[\widetilde{\Theta}_1\frac{|\nabla u|^2}{u^2}+\widetilde{\Theta}_2u^{\frac 2 n}\right]E_iu_{\overline i}+\operatorname{Re}\left[\widetilde{\Xi}_1\frac{|\nabla u|^2}{u^2}+\widetilde{\Xi}_2u^{\frac 2 n}+\widetilde{\Xi}_4\frac{n\sqrt{-1}u_0}{u}\right]\operatorname{Re}G_iu_{\overline i},
  \end{split}
\end{align}
where $a$ and $b$ are undertermined constants. The coefficients are:
$$\widetilde{\Delta}_1=(\beta+\frac{n+3}{n})d_1-\frac{n-1}{n}e_1+\frac 1 nd_4+(\beta+\frac{n+3}{n})a+2b,$$
$$\widetilde{\Delta}_2=-\frac 1 nd_1+(\beta+\frac{n+4}{n})d_2-\frac{n-1}{n}e_2+\frac 1 nd_4-a+b,$$
$$\widetilde{\Delta}_4=\frac 1 nd_1+(\beta+\frac{n+3}{n})d_4+\frac{n-1}{n}e_4+\frac 1 n\mu+a+(n\beta+n+2)b,$$
$$\widetilde{\Theta}_1=-\frac{n+2}{n}d_1+(\beta+\frac{n+2}{n})e_1+\frac 1 ne_4-(\beta+\frac{n+6}{n})a+2b,$$
$$\widetilde{\Theta}_2=-\frac 1 ne_1-\frac{n+2}{n}d_2+(\beta+\frac{n+3}{n})e_2+\frac 1 n e_4+\frac{n+1}{n}a+b,$$
$$\widetilde{\Xi}_1=\frac{n+2}{n}d_1-\frac{n-1}{n}e_1-\frac 1 n\mu+\frac 3 na+(n\beta+n+2)b,$$
$$\widetilde{\Xi}_2=\frac{n+2}{n}d_2-\frac{n-1}{n}e_2-\frac 1 n\mu-(n+1)b,$$
$$\widetilde{\Xi}_4=\frac{n+2}{n}d_4+\frac{n-1}{n}e_4-\beta\mu+2(n+1)b.$$

\vspace*{1em}
\textbf{Case $n\geqslant2$:} Let $\widetilde{\Delta}_l$, $\widetilde{\Theta}_l$ and $\widetilde{\Xi}_l$ be 0. Fix $d_l$, $\mu$, $a$, $b$, and solve $e_l$ from $\widetilde{\Xi}_1=\widetilde{\Xi}_2=\widetilde{\Xi}_4=0$:
$$e_1=\frac{(n+2)d_1-\mu+3a+n(n\beta+n+2)b}{n-1},$$
$$e_2=\frac{(n+2)d_2-\mu-n(n+1)b}{n-1},$$
$$e_4=\frac{-(n+2)d_4+n\beta\mu-2n(n+1)b}{n-1}.$$
Insert $e_l$ into $\widetilde{\Delta}_1-\widetilde{\Delta}_4$, then $\widetilde{\Delta}_1-\widetilde{\Delta}_4=\beta(d_1-d_4+a-2nb-\mu)$.

If $\beta=0$, we have $\widetilde{\Delta}_1=\widetilde{\Delta}_4$. Fix $d_1$, $\mu$, $a$, $b$, and solve $d_2$, $d_4$ from $\widetilde{\Delta}_1=\widetilde{\Delta}_2=0$:
$$d_2=d_1+na-n(n+1)b,~~d_4=-d_1-na+n^2b-\mu.$$
Insert $d_2$, $d_4$, $e_l$ and $\beta$ into $\widetilde{\Theta}_1$ and $\widetilde{\Theta}_2$, then
$$\widetilde{\Theta}_1=\frac{2[2(n+2)d_1+6a+n^2b]}{n(n-1)},~~\widetilde{\Theta}_2=\frac{2(n+2)[2d_1+(3n-1)a-n(3n+4)b]}{n(n-1)}.$$
Fix $b$, and solve $d_1$, $a$ from $\widetilde{\Theta}_1=\widetilde{\Theta}_2=0$: $\displaystyle d_1=-\frac{n(n+3)b}{2(n-1)}$, $\displaystyle a=\frac{n(n+1)b}{n-1}$. To ensure the positivity or negativity of the RHS of \eqref{high}, the coefficients of $\displaystyle\frac{|\nabla u|^2}{u^2}\sum_{i,j}|D_{ij}|^2$ and $\displaystyle u^{\frac 2 n}\sum_{i,j}|D_{ij}|^2$ must have the same sign, i.e. $d_1d_2\geqslant0$. Insert $d_1$ and $a$ into $d_2$: $\displaystyle d_2=\frac n 2b$, hence $b=0$. Similarly, the coefficients of $\displaystyle u^{\frac 2 n}\sum_{i,j}|E_{i\overline j}|^2$ and $\displaystyle \sum_i|G_i|^2$ must have the same sign, i.e. $e_2\mu\geqslant0$. Insert $d_2=b=0$ into $e_2$: $e_2=-\displaystyle\frac{\mu}{n-1}$, hence $\mu=0$. Now, all parameters are 0, and the identity \eqref{high} is trivial.

If $\beta\neq0$, then $d_4=d_1+a-2nb-\mu$. Insert $d_4$ and $e_l$ into $\widetilde{\Delta}_1$, $\widetilde{\Delta}_2$, and $\widetilde{\Delta}_4$:
$$\widetilde{\Delta}_1=\widetilde{\Delta}_4=\frac 1 n[(n\beta+2)d_1+(n\beta+n+1)a-(n\beta+n+2)nb],$$
$$\widetilde{\Delta}_2=\frac 1 n[(n\beta+2)d_2-(n-1)a+n^2b].$$

If $\beta\neq0$ and $\beta\neq-\displaystyle\frac 2 n$, $d_1$ and $d_2$ can be solved from $\widetilde{\Delta}_1=\widetilde{\Delta}_2=\widetilde{\Delta}_4=0$:
$$d_1=\frac{-(n\beta+n+1)a+(n\beta+n+2)nb}{n\beta+2},~~d_2=\frac{(n-1)a-n^2b}{n\beta+2}.$$
Insert $d_l$ and $e_l$ into $\widetilde{\Theta}_1$ and $\widetilde{\Theta}_2$:
$$\widetilde{\Theta}_1=\frac{(n\beta+n+4)[-2(n-1)a+(n\beta+2n+2)nb]}{n(n-1)},$$
$$\widetilde{\Theta}_2=-\frac{(n+2)(n\beta+4)[-2(n-1)a+(n\beta+2n+2)nb]}{n(n-1)(n\beta+2)},$$
then $a=\displaystyle\frac{(n\beta+2n+2)nb}{2(n-1)}$. To ensure the positivity or negativity of the RHS of \eqref{high}, the coefficients of $\displaystyle u^{\frac 2 n}\sum_{i,j}|D_{ij}|^2$, $\displaystyle u^{\frac 2 n}\sum_{i,j}|E_{i\overline j}|^2$ and $\displaystyle\sum_i|G_i|^2$ must have the same sign, i.e. $d_2$, $e_2$ and $\mu$ have the same sign. Insert $a$ into $d_2$ and $e_2$: $d_2=\displaystyle\frac n 2b$, $e_2=-\displaystyle\frac{2\mu+n^2b}{2(n-1)}$, hence $b=\mu=0$. Then, all parameters are 0, which means that the identity \eqref{high} is trivial again. 

\textbf{From discussions above, $\beta=-\displaystyle\frac 2 n$ is the only possible case when $n\geqslant 2$.}

When $\beta=-\displaystyle\frac 2 n$, rewrite $d_4$ and $e_l$:
$$d_4=d_1-\mu+a-2nb,~~e_1=\frac{(n+2)d_1-\mu+3a+n^2b}{n-1},$$
$$e_2=\frac{(n+2)d_2-\mu-n(n+1)b}{n-1},~~e_4=\frac{-(n+2)d_1+n\mu-(n+2)a+2nb}{n-1}.$$
Insert them and $\beta=-\displaystyle\frac 2 n$ into $\widetilde{\Delta}_l$ and $\widetilde{\Theta}_l$:
$$\widetilde{\Delta}_1=-\widetilde{\Delta}_2=\widetilde{\Delta}_4=-\frac{n-1}{2(n+2)}\widetilde{\Theta}_1=\frac{n-1}{n}a-nb,$$
$$\widetilde{\Theta}_2=\frac{(n+2)}{n(n-1)}[-2d_1+2d_2+(n-3)a-n^2b].$$
Fix $d_1$ and $a$, and solve $d_2$ and $b$ from $\widetilde{\Delta}_1=\widetilde{\Theta}_2=0$: $d_2=d_1+a$, $b=\displaystyle\frac{n-1}{n^2}a$. Then
$$\beta=-\frac 2 n,~~b=\frac{n-1}{n^2}a,~~d_2=d_1+a,~~d_4=d_1-\frac{n-2}{n}a-\mu,~~e_1=\frac{(n+2)(d_1+a)-\mu}{n-1},$$
$$e_2=\frac{(n+2)d_1+(2+\displaystyle\frac 1 n)a-\mu}{n-1},~~e_4=\frac{-(n+2)d_1-(n+\displaystyle\frac 2 n)a+n\mu}{n-1}.$$

Rewrite the identity \eqref{high} with the parameters above as the following proposition.
This identity has three undetermined parameters $\{d_1, a, \mu\}$.

\begin{proposition}\label{Jerison-LeeB}
  For $n\geqslant2$. The only positive $\{(0,0),2,6,+\}$ type identity is
  \begin{align}
    \begin{split}\label{total}
      &u^{\frac 2 n}\operatorname{Re}\Big\{u^{-\frac 2 n}\Big\{\Big(d_1\frac{|\nabla u|^2}{u}+(d_1+a)u^{\frac{n+2}{n}}+(d_1-\frac{n-2}{n}a-\mu)n\sqrt{-1}u_0\Big)D_i\\
      &+\Big(\frac{(n+2)(d_1+a)-\mu}{n-1}\frac{|\nabla u|^2}{u}+\frac{(n+2)d_1+(2+\displaystyle\frac 1 n)a-\mu}{n-1} u^{\frac{n+2}{n}}\\
      &\quad\quad+\frac{-(n+2)d_1-(n+\displaystyle\frac 2 n)a+n\mu}{n-1}\cdot n\sqrt{-1}u_0\Big)E_i-\mu n\sqrt{-1}u_0G_i\\
      &+a\Big[D_ju_{\overline j}-E_ju_{\overline j}+\frac{n-1}{n^2}\Big(\frac{|\nabla u|^4}{u^2}+u^{\frac 2 n}|\nabla u|^2-n\frac{|\nabla u|^2\cdot n\sqrt{-1}u_0}{u}\\
      &\quad\quad+(n+1)u^{\frac{n+2}{n}}\cdot n\sqrt{-1}u_0-(n+1)n^2 u_0^2\Big)\Big]\frac{u_i}{u}\Big\}\Big\}_{,\overline i}\\
      =&\left[d_1\frac{|\nabla u|^2}{u^2}+(d_1+a)u^{\frac 2 n}\right]\sum_{i,j}|D_{ij}|^2+(d_1+2a)\sum_i|D_i|^2\\
      &+\left[\frac{(n+2)(d_1+a)-\mu}{n-1}\frac{|\nabla u|^2}{u^2}+\frac{(n+2)d_1+(2+\displaystyle\frac 1 n)a-\mu}{n-1}u^{\frac 2 n}\right]\sum_{i,j}|E_{i\overline j}|^2\\
      &+\frac{(n+2)d_1+3a-\mu}{n-1}\sum_i|E_i|^2+\mu\sum_i|G_i|^2+\frac{(2n+1)d_1+3a-\mu}{n-1}\operatorname{Re}D_iE_{\overline i}\\
      &+(-d_1+\frac{n-2}{n}a+\mu)\operatorname{Re}D_iG_{\overline i}+\frac{(n+2)d_1+(n+\displaystyle\frac 2 n)a-n\mu}{n-1}\operatorname{Re}E_iG_{\overline i}.
    \end{split}
  \end{align}
  The parameters $d_1$, $a$, and $\mu$ satisfy
  \begin{equation}\label{con}
    d_1\geqslant\max\{0,-a\},~~(n+2)d_1-\mu\geqslant\max\{-(n+2)a,-(2+\frac 1 n)a\},
  \end{equation}
  and the following matrix is semi-positive:
  \begin{equation}\label{matrix}
    \begin{pmatrix}
      \mu & \displaystyle\frac 1 2(-d_1+\frac{n-2}{n}a+\mu) & \displaystyle\frac{(n+2)d_1+(n+\frac 2 n)a-n\mu}{2(n-1)}\\[10pt]
      \displaystyle\frac 1 2(-d_1+\frac{n-2}{n}a+\mu) & 2(d_1+a) & \displaystyle\frac{(2n+1)d_1+3a-\mu}{2(n-1)}\\[10pt]
      \displaystyle\frac{(n+2)d_1+(n+\frac 2 n)a-n\mu}{2(n-1)} & \displaystyle\frac{(2n+1)d_1+3a-\mu}{2(n-1)} & \displaystyle\frac{2n-1}{(n-1)^2}[(n+2)d_1+\frac{n^2+5n-3}{2n-1}a-\mu]
    \end{pmatrix}.
  \end{equation}
\end{proposition}

\begin{proof}
  The coefficients of $\displaystyle\frac{|\nabla u|^2}{u^2}\sum_{i,j}|D_{ij}|^2$, $\displaystyle u^{\frac 2 n}\sum_{i,j}|D_{ij}|^2$, $\displaystyle\frac{|\nabla u|^2}{u^2}\sum_{i,j}|E_{i\overline j}|^2$ and $\displaystyle u^{\frac 2 n}\sum_{i,j}|E_{i\overline j}|^2$ are non-negative, i.e. \eqref{con}. By Lemma \ref{Cauchy}, the RHS of identity is greater than or equal to a quadratic form with \eqref{matrix} as matrix.
\end{proof}

Now we prove Theorem \ref{JL2}.
\begin{proof}[Proof of Theorem \ref{JL2}]
From Prop.\ref{Jerison-LeeB}. we know that three constants $d_1$, $a$, and $\mu$ determine a three-dimensional family of differential identities as Jerison-Lee stated.
\end{proof}

Three constants $d_1$, $a$, and $\mu$ determine a three-dimensional family of differential identities as Jerison-Lee stated. If $d_1=1$, $a=0$ and $\mu=3$, we yield the classical Jerison-Lee identity (4.2) in \cite{MR0924699}:
\begin{align*}
  &u^{\frac 2 n}\operatorname{Re}\left\{u^{-\frac 2 n}\left[\left(\frac{|\nabla u|^2}{u}+u^{\frac{n+2}{n}}\right)(D_i+E_i)-n\sqrt{-1}u_0(2D_i-2E_i+3G_i)\right]\right\}_{,\overline i}\\
  =&\left[\frac{|\nabla u|^2}{u^2}+u^{\frac 2 n}\right]\sum_{i,j}(|D_{ij}|^2+|E_{i\overline j}|^2)+\sum_{i}(|D_i|^2+|E_i|^2+3|G_i|^2)\\
  &+2\operatorname{Re}D_iE_{\overline i}+2\operatorname{Re}D_iG_{\overline i}-2\operatorname{Re}E_iG_{\overline i}\\
  =&u^{\frac 2 n}\sum_{i,j}(|D_{ij}|^2+|E_{i\overline j}|^2)+\sum_i(|G_i|^2+|G_i+D_i|^2+|G_i-E_i|^2)+u^{-2}\sum_{i,j,k}|D_{ij}u_{\overline k}+E_{i\overline k}u_j|^2.
\end{align*}
If $d_1=0$, $a=n$ and $\mu=n+2$, we yield the identity (4.3) in \cite{MR0924699}, which is also positive:
\begin{align*}
  &u^{\frac 2 n}\operatorname{Re}\Big\{u^{-\frac 2 n}\Big\{\Big(nu^{\frac{n+2}{n}}-2n^2\sqrt{-1}u_0\Big)D_i+\Big((n+2)\frac{|\nabla u|^2}{u}+u^{\frac{n+2}{n}}+2n\sqrt{-1}u_0\Big)E_i\\
  &-(n+2)n\sqrt{-1}u_0G_i+n\Big[D_ju_{\overline j}-E_ju_{\overline j}+\frac{n-1}{n^2}\Big(\frac{|\nabla u|^4}{u^2}+u^{\frac 2 n}|\nabla u|^2\\
  &\quad\quad-n\frac{|\nabla u|^2\cdot n\sqrt{-1}u_0}{u}+(n+1)u^{\frac{n+2}{n}}\cdot n\sqrt{-1}u_0-(n+1)n^2 u_0^2\Big)\Big]\frac{u_i}{u}\Big\}\Big\}_{,\overline i}\\
  =&nu^{\frac 2 n}\sum_{i,j}|D_{ij}|^2+2n\sum_i|D_i|^2+\left[(n+2)\frac{|\nabla u|^2}{u^2}+u^{\frac 2 n}\right]\sum_{i,j}|E_{i\overline j}|^2\\
  &+2\sum_i|E_i|^2+(n+2)\sum_i|G_i|^2+2\operatorname{Re}D_iE_{\overline i}+2n\operatorname{Re}D_iG_{\overline i}-2\operatorname{Re}E_iG_{\overline i}\\
  =&(n+2)\frac{|\nabla u|^2}{u^2}\sum_{i,j}|E_{i\overline j}|^2+\sum_i|E_i|^2+(n-2)\sum_i|D_i|^2+(n+1)\sum_i|G_i+D_i|^2\\
  &+\sum_i|G_i-D_i-E_i|^2+u^{\frac 2 n}\sum_{i,j}(|E_{i\overline j}|^2+n|D_{ij}|^2).
\end{align*}
if $d_1=1$, $a=0$ and $\mu=3n$, we yield the identity (4.4) in \cite{MR0924699}, which is not positive:
\begin{align*}
  &u^{\frac 2 n}\operatorname{Re}\Big\{u^{-\frac 2 n}\Big\{\Big(\frac{|\nabla u|^2}{u}+u^{\frac{n+2}{n}}\Big)(D_i-2E_i)\\
  &\quad\quad\quad\quad\quad-n\sqrt{-1}u_0[(3n-1)D_i-(3n+2)E_i+3n G_i]\Big\}\Big\}_{,\overline i}\\
  =&\left[\frac{|\nabla u|^2}{u^2}+u^{\frac 2 n}\right]\sum_{i,j}(|D_{ij}|^2-2|E_{i\overline j}|^2)+\sum_i(|D_i|^2-2|E_i|^2+3n|G_i|^2)\\
  &-\operatorname{Re}D_iE_{\overline i}+(3n-1)\operatorname{Re}D_iG_{\overline i}-(3n+2)\operatorname{Re}E_iG_{\overline i}.
\end{align*}

Notice that the matrix \eqref{matrix} can't be semi-positive if $\mu=0$. W.L.O.G., assume that $\mu=3$, then the positivity condition \eqref{con} and matrix \eqref{matrix}$\geqslant0$ determine the range for $d_1$ and $a$, which can be described by the following figure:

\begin{figure}[h]
  \centering
  \includegraphics[scale=0.4]{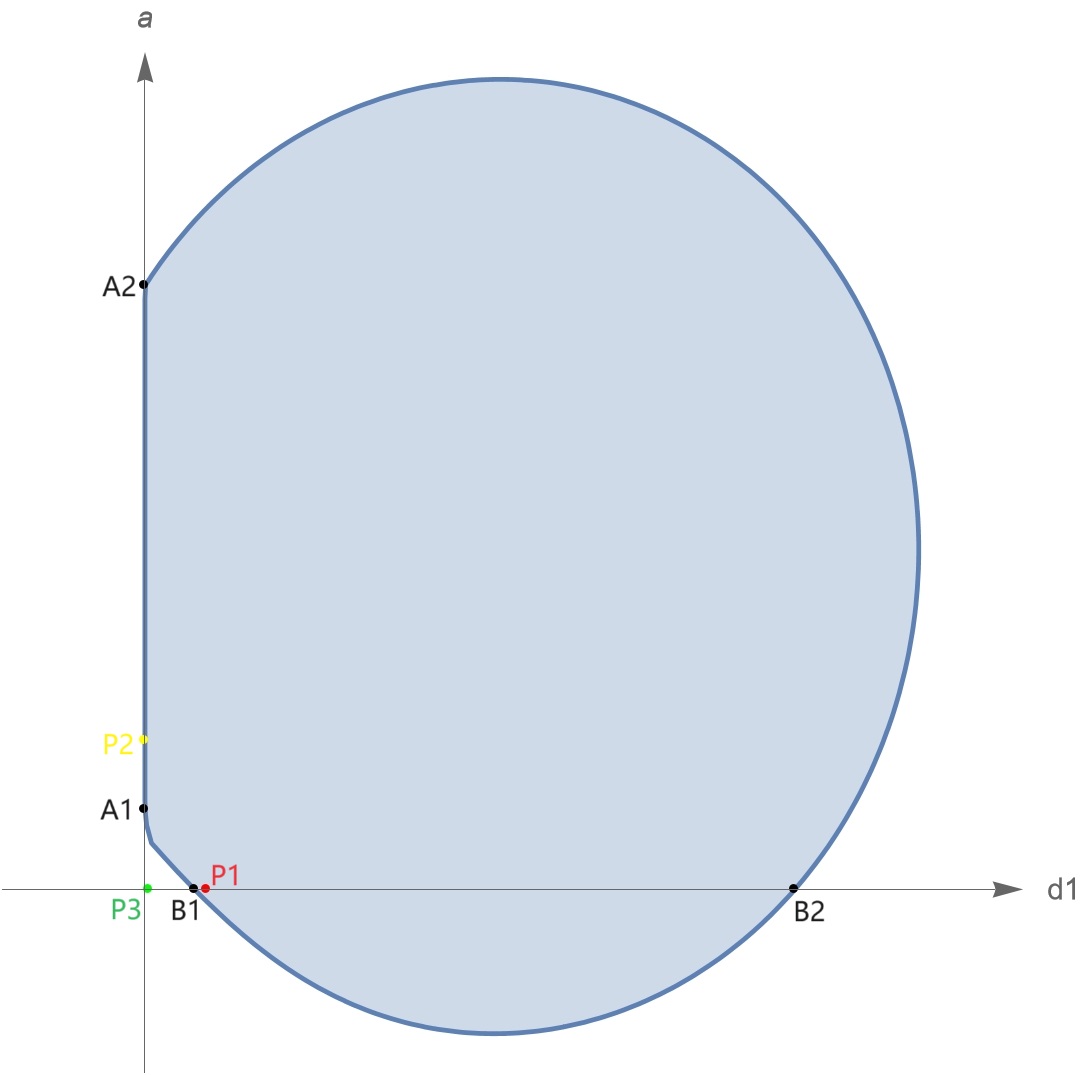}
  \caption{The range for $d_1$ and $a$ when identity \eqref{total} is positive.}
\end{figure}

The coordinates of key points are:
$${\color{red}\text{P1}}=(1,0),\quad{\color{yellow}\text{P2}}=(0,\frac{3n}{n+2}),\quad{\color{green}\text{P3}}=(\frac1 n,0),\quad\text{A1}=(0,\frac{3n}{2n+1}),$$
\begin{align*}
  \text{A2}=&\Big(0,\frac{2\sqrt{n(73n^7+538n^6+1435n^5+134n^4-1439n^3-120n^2+292n+48)}}{3n^4-2n^3-5n^2+26n+8}\\
  &\hspace*{1em}\times\cos\Big\{\frac 1 3\arccos\Big[\sqrt n(595n^{10}+7017n^9+30666n^8+55019n^7-7692n^6-
  82095n^5\\
  &\hspace*{10em}-12345n^4+38598n^3+2556n^2-6920n-1440)\\
  &\hspace*{3em}\times(73n^7+538n^6+1435n^5+134n^4-1439n^3-120n^2+292n+48)^{-3/2}\Big]\Big\}\\
  &\hspace*{1em}+\frac{n(10n^3+35n^2+4)}{(n+2)(3n^3-8n^2+11n+4)}\Big),
\end{align*}
$$\text{B1}=\left(\frac{\sqrt{468n^4+1380n^3+n^2-1500n+612}}{3n^2+8n+4}\cos\frac{\theta-2\pi}{3}+\frac{24n^2+43n-18}{2(n+2)(3n+2)},0\right),$$
$$\text{B2}=\left(\frac{\sqrt{468n^4+1380n^3+n^2-1500n+612}}{3n^2+8n+4}\cos\frac{\theta}{3}+\frac{24n^2+43n-18}{2(n+2)(3n+2)},0\right),$$
$$\theta=\arccos\frac{9936n^6+44172n^5+32202n^4-66149n^3-35622n^2+54756n-15336
}{(468n^4+1380n^3+n^2-1500n+612)^{\frac 3 2}},$$
where {\color{red}P1}, {\color{yellow}P2}, {\color{green}P3} correspond with identity (4.2), (4.3), (4.4) in \cite{MR0924699}, and A1, A2, B1, B2 are intersections of coordinate axes and boundary of the range. From the figure, it's obvious that (4.2) and (4.3) are positive, and (4.4) is not positive.

\vspace*{1em}
\textbf{Case $n=1$:} It's easy to check that the identity \eqref{total} degenerates to classical Jerison-Lee identity (4.2) in \cite{MR0924699} when $n=1$, hence identity (4.2), (4.3) and (4.4) in \cite{MR0924699} are identical.
In fact, it's the only possible $\{(0,0),2,6,+\}$ type identity when $n=1$.

\begin{proposition}
  For $n=1$, the only positive $\{(0,0),2,6,+\}$ type identity is the classical Jerison-Lee identity (4.2) in \cite{MR0924699}.
\end{proposition}

\begin{proof}
    Notice that \eqref{1'} degenerates to $u^{-\beta}[u^{\beta-1}|\nabla u|^2D_1]_{,\overline 1}$, which is the $d_1$ term in the vector field of \eqref{2'}. Hence we assume that $a=0$. Rewrite \eqref{high} as
    \begin{align}
      \begin{split}\label{low}
        &\left[\eqref{2'}+b\times\eqref{3'}\right]\Big|_{n=1}\\
        =&\left[d_1\frac{|\nabla u|^2}{u^2}+d_2u^{\frac 2 n}\right]\sum_{i,j}|D_{ij}|^2+d_1\sum_i|D_i|^2+\mu\sum_i|G_i|^2-d_4\operatorname{Re}D_iG_{\overline i}\\
        &+\operatorname{Re}\left[\widetilde{\Delta}_1\frac{|\nabla u|^2}{u^2}+\widetilde{\Delta}_2u^{\frac 2 n}+\widetilde{\Delta}_4\frac{n\sqrt{-1}u_0}{u}\right]\operatorname{Re}D_iu_{\overline i}\\
        &+\operatorname{Re}\left[\widetilde{\Xi}_1\frac{|\nabla u|^2}{u^2}+\widetilde{\Xi}_2u^{\frac 2 n}+\widetilde{\Xi}_4\frac{n\sqrt{-1}u_0}{u}\right]\operatorname{Re}G_iu_{\overline i},
      \end{split}
    \end{align}
    The coefficients are:
    $$\widetilde{\Delta}_1=(\beta+4)d_1+d_4+2b,~~\widetilde{\Delta}_2=-d_1+(\beta+5)d_2+d_4+b,$$
    $$\widetilde{\Delta}_4=d_1+(\beta+4)d_4+\mu+(\beta+3)b,~~\widetilde{\Xi}_1=3d_1-\mu+(\beta+3)b,$$
    $$\widetilde{\Xi}_2=3d_2-\mu-2b,~~\widetilde{\Xi}_4=3d_4-\beta\mu+4b.$$
    
    Fix $\beta$, $\mu$, $b$, and solve $d_1$, $d_2$, $d_4$ from $\widetilde{\Xi}_1=\widetilde{\Xi}_2=\widetilde{\Xi}_4=0$:
    $$d_1=\frac{\mu+(\beta+3)b}{3},~~d_2=\frac{\mu+2b}{3},~~d_4=\frac{\beta\mu+4b}{3}.$$
    Insert them into $\widetilde{\Delta}_l$:
    $$\widetilde{\Delta}_1=\frac{2(\beta+2)\mu+(\beta^2+7\beta+22)b}{3},$$
    $$\widetilde{\Delta}_2=\frac{2(\beta+2)\mu+(\beta+14)b}{3},$$
    $$\widetilde{\Delta}_4=\frac{(\beta+2)^2\mu+4(2\beta+7)}{3}.$$
    Consider $\widetilde{\Delta}_1-\widetilde{\Delta}_2=0$, i.e. $(\beta+2)(\beta+4)b=0$.
    
    If $\beta=-4$, then $\displaystyle\widetilde{\Delta}_1=\frac{-4\mu+10b}{3}=0$, $\displaystyle\widetilde{\Delta}_4=\frac{4\mu-4b}{3}=0$, hence $\mu=b=0$, then $d_1=d_2=d_4=0$, which means that the identity is trivial.
    
    If $\beta\neq-2$ and $\beta\neq-4$, $b=0$, $\mu=0$, then identity is trivial as well.

    \textbf{From discussions above, $\beta=-2$ is the only possible case when $n=1$.}
    
    If $\beta=-2$, then $\widetilde{\Delta}_2=0$ yields that $b=0$. All parameters are:
    $$d_1=d_2=\frac\mu 3,~~d_4=-\frac 2 3\mu,~~\beta=-2,~~b=0,$$
    which is just identical to the classical Jerison-Lee identity case.
\end{proof}

\vspace*{1em}
All possible identities can be found by dimensional conservation and invariant tensors, and then a linear combination of them with some appropriate parameters will always work. For the 'near-critical' subcritical exponent case in CR geometry, the best choice is to use the same structure of identity as the critical exponent case. For example, the vector fields in Case 2 in Section 3 are composed of invariant tensors. However, we need a little change with \eqref{3} for Case 3 in Section 3, but we don't break the continuity when $\alpha\to3$. For the case $\alpha$ far away from the critical exponent, without caring about the critical exponent case, our choices of parameters are flexible, such as Case 1 and Case 4 in Section 3.
\bibliographystyle{plain}
\bibliography{reference.bib}
\end{document}